\documentclass[12pt]{article}
\usepackage{epsfig,graphicx,subfigure}
\usepackage{amsmath,amssymb,amsthm} 
\usepackage{epstopdf}
\begin{document}
%
%
\newtheorem{theorem}{Theorem}
\newtheorem{lemma}{Lemma}
\newtheorem{proposition}{Proposition}
\newtheorem{definition}{Definition}
\newtheorem{corollary}{Corollary}

\numberwithin{equation}{section}
\numberwithin{theorem}{section}
\numberwithin{proposition}{section}
\numberwithin{lemma}{section}
\numberwithin{definition}{section}
\numberwithin{corollary}{section}
\numberwithin{figure}{section}
\numberwithin{table}{section}

\newcommand{\beqn}{\begin{equation}}
\newcommand{\eeqn}{\end{equation}}
\newcommand{\nn}{\nonumber}
\newcommand{\la}{\langle}
\newcommand{\ra}{\rangle}
\newcommand{\cleq}{{\preccurlyeq}}

\def\R{{\mathbb R}}
\def\C{{\mathbb{C}}}
\def\pa {{\partial}}
\def\ep{{\epsilon}}
\newcommand{\ve}{\varepsilon}

\def\n{{\bf n}}
\newcommand{\grad}{{\nabla} }

\def\M{{\bf M}}
\def\L{{\cal L}}
\def\m{{\bf m}}
\def\mt{{\tilde{m}}}
\def\mtb{{\bf \tilde{m}}}
\def\Mo{{ \overline M}}
\def\bMo{{ \bf \Mo}}
\def\Mt{{ {\tilde {\bf M}} }}
\def\betat{{ {\tilde \beta} }}

\def\B{{\cal B}}
\def\Bb{{ \overline \B}}
\def\E{{\textbf E}}
\def\P{{\textbf P}}

\def\F{{\cal F}}
\def\T{{\cal T}}
\def\S{{\cal S}}
\def\St{{\overline{\S}}}
\def\Tt{{\overline{\T}}}
\def\Qt{{\overline{Q}}}

\def\D{{\tilde{D}}}
\def\Om{{ \tilde{\Omega}}}

\def\f{{\bf f}}
\def\g{{\bf g}}
\def\p{{\bf p}}
\def\q{{\bf q}}
\def\po{{ \bf \overline{p} }}
\def\qo{{\bf \overline{q}}}
\def\v{{\bf v}}
\def\h{{\bf h}}
\def\a{{\bf a}}
\def\th{{\bf \tilde{h}}}

\def\CdZ{{ \dot{C}^1[0,Z]}}

\title{Determining the twist in an optical fiber}

\author{
Rakesh\thanks{Rakesh was partially supported by NSF grants DMS 0907909, DMS 1312708.}\\
Department of Mathematical Sciences\\
University of Delaware\\
Newark, DE 19716, USA\\
~
\and
Jiahua Tang\\
701 First Avenue\\
Sunnyvale, CA 94089, USA\\
~
\and
Andrew A Lacey\\
School of Mathematical and Computer Sciences,\\
Heriot-Watt University\\
Riccarton, Edinburgh, EH14 4AS, UK}

\date{December 7, 2015}

\maketitle

\begin{abstract}
We determine the twist in a birefringent optical fiber from measurements, at one end of the 
fiber, of the fiber response to an impulsive source at the same end. This is the inverse problem of determining a non-constant 
coefficient, of a first 
order hyperbolic system in one space dimension with two speeds of propagation, from measurements at one end of an interval, of the 
solution of this system corresponding to an impulsive source at the same end. We prove a stability result for this inverse problem and 
give a provable reconstruction algorithm for this inverse problem.

\end{abstract}


%
\section{Introduction}
\noindent
We determine the twist in a birefringent optical fiber from measurements, at one end of the 
fiber, of the fiber response to an impulsive source at the same end.

Consider a birefringent fiber stretching along the $z$ axis, with two channels with different but constant speeds of propagation
 twisting around 
each other with the twist captured by a real valued function $\beta(z)$ on $[0,\infty)$ with $\beta(0)=0$, $\beta'(0)=0$.
The fiber is 
probed by an impulsive source from the left 
end, and the fiber response is measured at the same end (see  Figure \ref{fmodel}). The goal is to determine the twist $\beta(z)$ for 
$z>0$ from the fiber response.
\begin{figure}[!h]
\centering
\epsfig{file=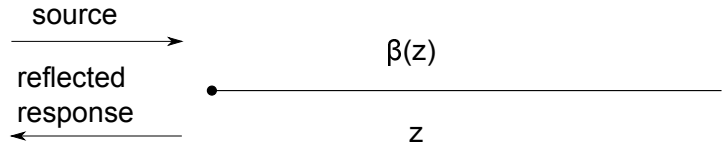, height=0.7in} 
\caption{Fiber model}
\label{fmodel}
\end{figure}

The forward problem was modeled in \cite{mpi2000} and we reproduce this 
derivation in section \ref{sec:model} since it is not readily available. 
Without loss of generality, we assume the two channels have speeds $c$ and $1$ with ${0<c<1}$ and the four 
component vector function ${\bf M}(z,t)$ represents the signal at position $z$ at time $t$ with the $M_1, M_3$ components denoting 
the left moving waves of speeds $1$ and $c$ respectively, and $M_2, M_4$ components the right moving waves of speeds $1$ and $c$ 
(see Figure \ref{lrwave}).
\begin{figure}[!h]
\centering
\epsfig{file=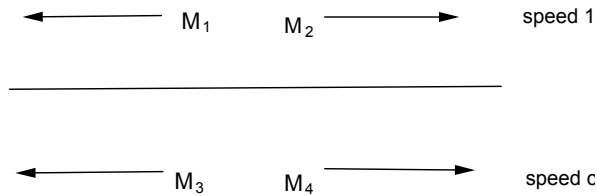, height=1.0in}
\caption{Left and right moving waves}
\label{lrwave}
\end{figure}
The propagation and the reflection of the impulsive source in the twisted fiber is modeled by the solution of the initial boundary
value problem (IBVP) for the  hyperbolic system of PDEs
\begin{subequations}
\begin{align}
{\bf M}_t&=A{\bf M}_z+\beta B{\bf M}, \quad z\ge 0,\quad t\in {\mathbb R}, \label{IVP11} \\
M_2(0,t)&=\delta(t), \quad M_4(0,t)=0,\quad  t \in \R, \label{delta1}\\
{\bf M}(z,t)&={\bf 0},\quad t<0,\quad z\ge 0 \label{IVP21}
\end{align}
\end{subequations}
where 
\beqn
A=\begin{bmatrix}
1 & 0 & 0 & 0\\
0 & -1 & 0 & 0\\
0 & 0 & c & 0\\
0 & 0 & 0 &-c
\end{bmatrix},
\quad 
B=\frac{1}{2}\begin{bmatrix}
0 & 0 & -1-c & -1+c\\
0 & 0 & 1-c & 1+c\\
1+c & -1+c & 0 & 0\\
1-c & -1-c & 0 & 0
\end{bmatrix}. \label{CDcd}
\eeqn
The well-posedness of the IBVP (\ref{IVP11}) - (\ref{IVP21}), for reasonably smooth $\beta(\cdot)$, is stated in Theorem 
\ref{wellposed}.  
The initial and boundary conditions represent a plane wave sent from the left end of the fiber {\bf along the faster channel}
and our goal is to recover the twist $\beta(z)$ given the fiber response, $M_1(0,t)$ and $M_3(0,t)$,
at the left end.

An analysis of the linearization of the map $\beta(\cdot) \to [M_1(0, \cdot), M_3(0, \cdot)]$, around $\beta=0$, is 
instructive. Since the solution of (\ref{IVP11})-(\ref{IVP21}) corresponding $\beta=0$ is $[0, \delta(t-z),0, 0]$, the 
linearization of the above map around $\beta=0$ is the map
\[
d \beta  \to [ dM_1(0, \cdot), dM_3(0, \cdot)]
\]
where $d\M(z,t)$ is the solution of the IBVP
\begin{align*}
(d{\bf M})_t-A(d{\bf M})_z=(d\beta) B [0, \delta(t-z),0, 0],  & \qquad (z,t) \in [0, \infty) \times \R 
\\
dM_2(0,t)=0, ~~dM_4(0,t)=0, & \qquad t \in \R
\\
d{\bf M}(z,t)=0, & \qquad t<0, ~z \in [0, \infty).
\end{align*}
Solving this IBVP one obtains that $ dM_1(0,t)=0$ and
\begin{align}
(d M)_3(0,t)=\frac{c-1}{2(c+1)}(d\beta)\left(\frac{ct}{1+c}\right)H(t)
\label{eq:dm3}
\end{align}
where $H(t)$ is the Heaviside function, so the linearization of the above mentioned map is
\[
d \beta (z) \to [0, dM_3(0,t)]
\]
with $dM_3(0,t)$ given by (\ref{eq:dm3}). The analysis of this linearized map suggests that, for the original (nonlinear) 
problem, to recover $\beta(z)$ on $[0,Z]$ one may need only $M_3(0,t)$ for all $t$ in $[0, Z(1+c)/c]$. 

Unfortunately, our results do not meet our expectations because our results require knowledge of both 
$M_1(0,t)$ and $M_3(0,t)$. Theorem \ref{stab} gives a stability result (and hence a uniqueness result) for the inverse problem and 
Theorem \ref{thm:recon} 
asserts that  we can reconstruct $\beta( \cdot)$ if we are given both $M_1(0, \cdot)$ and $M_3(0, \cdot)$ and an upper bound on
the $L^2$ norm of $\beta$.

Below $l \cleq r$ will mean $l \leq C r$ for some constant $C$, we define the operator
\begin{align*}
{\cal L} :=I\partial_t-A\partial_z-\beta B
\end{align*} 
and, for any $Z>0$, we define
\[
\dot{C}^1[0,Z] = \{ \beta \in C^1[0,Z] \, | \, \beta(0)=0, \, \beta'(0)=0 \}, \qquad Y = \frac{ 2cZ}{1+c}.
\]

Our first result addresses the well-posedness of the IBVP (\ref{IVP11})-(\ref{IVP21}).
\begin{theorem}[Well-posedness]\label{wellposed}
If $\beta \in \dot{C}^1[0,Z]$ then (\ref{IVP11})-(\ref{IVP21}) has a unique 
solution\footnote{on the region $\{(z,t) \, : \, 0 \leq z \leq Z, ~ z+t \leq 2Z \}$ - see Figure \ref{fig:maxnorm}}
\[
\M(z,t) = \delta(t-z) [0,1,0,0] + \m(z,t) H(t-z)
\]
where $\m(z,t)$ is the unique $C^1$ solution of the characteristic boundary value problem (CBVP)
\begin{subequations}
\begin{align}
\m_t&=A \m_z + \beta B \m, ~\qquad \text{on} ~0 \leq z \leq t \leq 2Z-z, \label{compt} \\
m_2(0,t)&=m_4(0,t)=0,\quad 0 \le t \le 2Z, \label{eq:m1bc}\\
m_1(z,z)=0,\quad m_3(z,z)&=\frac{c-1}{2(c+1)}\beta(z),\quad m_4(z,z)=\frac{c+1}{2(c-1)}\beta(z),\quad 0 \leq z \leq Z. 
\label{compb}
\end{align}
\end{subequations}
\end{theorem}
Theorem \ref{wellposed} is valid only for those $\beta$ with $\beta(0)=0, ~ \beta'(0)=0$ - see the definition of
 $\dot{C}^1[0,Z]$, because these are 
 forced\footnote{Use (\ref{compt})-(\ref{compb}) for $m_3, m_4$.
The condition $\beta(0)=0$ is natural because it represents an untwisted fiber at the $z=0$ end. 
The condition $\beta'(0)=0$ is not natural 
and perhaps could be avoided if we work with $\beta$ in the optimal regularity class, but that is unknown.} 
by the matching conditions if $\m$ is to be $C^1$. 

The methods in this article can be modified to show that if $\beta \in L^2[0,Z]$ then the 
CBVP (\ref{compt})-(\ref{compb}) 
has a weak solution which is locally 
$L^2$ on the region $t \geq z \geq 0$ and has local $L^2$ traces on lines parallel to the $z$ or the $t$ axes. 
This would
be needed 
for a complete solution of our inverse problem but we do not prove this result here because we are unable to
complete 
other parts of the solution of this inverse problem, as explained later.

Since $M_1(0,t)$ and $M_3(0,t)$ are zero for $t<0$ and $M_1(0,t) = m_1(0,t)$, $M_3(0,t) = m_3(0,t)$ for $t \geq 0$,
we will freely switch between $M_1(0,\cdot), M_3(0, \cdot)$ and $m_1(0, \cdot), m_3(0, \cdot)$ on the interval $[0, \infty)$.

Our next result shows that if the source is initiated in the fast channel then the reflected boundary data from both channels, 
over the time interval $[0,2Z]$, is enough to 
stably distinguish the twist function $\beta(z)$, up to a depth $Y$. Note that, a signal originating at $z=0$ at time $t=0$, traveling at 
the fast speed $1$, and reflected at $z=Y$ with the slower speed $c$, will just make it back to $z=0$ at time $t=2Z$ - see Figure 
\ref{fig:maxnorm}.
\begin{figure}[!h]
\centering
\epsfig{file=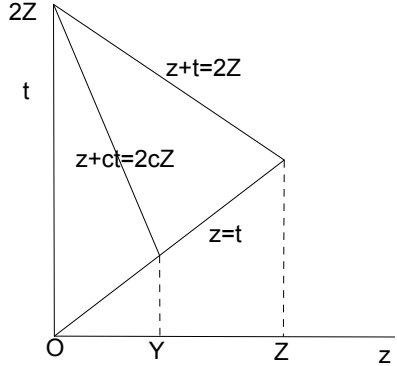, height=1.9in}
\caption{Depth sensed in time $2Z$}
\label{fig:maxnorm}
\end{figure}
This theorem is relevant for the reconstruction of $\beta(\cdot)$ from the data.
For arbitrary $ K, Z>0$  define
\begin{align*}
\B_K &:=\{\beta \in \dot{C}^1[0,Z]~ |~ \|\beta(\cdot)\|_{L^2[0,Z]}^2\le K\}.
\end{align*}

\begin{theorem}[Injectivity and Stability] \label{stab}
If ${\bf M}, {\tilde {\bf M}}$ are the solutions of (\ref{IVP11})-(\ref{IVP21}) corresponding to 
$\beta, {\tilde \beta}\in \B_K$ then 
\begin{align*}
\|(\beta-{\tilde \beta})(\cdot)\|_{L^2[0,Y]}^2
\, \cleq \, 
\|(M_1- {\tilde M}_1)(0,\cdot)\|_{L^2[0,2Z]}^2+\|(M_3- {\tilde M}_3)(0,\cdot)\|_{L^2[0,2Z]}^2
\end{align*}
where the constant depends only on $c,Z$ and $K$.
\end{theorem}

Define the forward (nonlinear) map
\begin{align*}
{\F}: \CdZ &\rightarrow  C^1[0,2Z]\times C^1[0,2Z],\\
\beta(z)&\mapsto [m_1(0,t), m_3(0,t)]
\end{align*}
which maps the coefficient to the full reflection data. Theorem \ref{stab} guarantees that ${\F}$ is injective 
and ${\F}^{-1}$ is continuous in the appropriate norms, at least when $\beta$ is restricted to the interval $[0,Y]$. 
Our main goal is to invert $\F$ and we state our result in the 
following  theorem. Again note that given $m_1(0,t), m_3(0,t)$ over $[0,2Z]$, one recovers $\beta( \cdot)$ only on 
$[0, Y]$ and not on the whole interval $[0,Z]$.
\begin{theorem}[Reconstruction] \label{thm:recon}
If $\beta \in \dot{C}^1[0,Z]$ and $\m(z,t)$ is the corresponding solution of  (\ref{compt})-(\ref{compb}) then
given $(m_1(0,t), m_3(0,t))$ for all $t \in [0,2Z]$, one can reconstruct $\beta(\cdot)$ over the interval $[0,Y]$, if 
an upper bound on $\|\beta\|_{L^2[0,Y]}$ is also provided. 
\end{theorem}

Along with the inversion of $\F$, it is important to characterize the range of $\F$. Necessary conditions similar to those in \cite{BI02} 
may be derived but they are far from sufficient for our problem.  Actually $\dot{C}^1[0,Z]$ is not the appropriate domain for $\F$ and 
the optimal answer will be obtained by studying the inversion and the range characterization of the map
 $\beta \to m_3(0, \cdot)$ rather than that of $\F$. We expect $L^2[0,Z]$ to be best suited for the domain of these maps. In our problem, 
the medium is probed by a 
source wave traveling at the faster speed - see the boundary conditions (\ref{delta1}). It would be interesting to 
also study the problem when the source wave travels at the slower speed (the boundary conditions
are changed to $M_2(0,t)=0, ~ M_4(0,t) = \delta(t)$). Unfortunately, we have no results for this case because of the complications due to the presence of precursor 
waves as noted by Belishev in his work; we will say more about this in the literature review next.

Inverse problems for hyperbolic PDE, in one space dimension, with a single speed of propagation, have been 
studied by Gelfand, Levitan, Marchenko, Krein, Blagoveschentski and many others; \cite{Burridge1980305} and Browning's 
thesis \cite{brown} contain a thorough survey of these results. Inverse problems for hyperbolic systems, in one space 
dimension,
with multiple speeds of propagation have been studied by Belishev and his collaborators (see \cite{belishev1997two}, 
\cite{BI02}, \cite{BI03} and 
specially \cite{Belsurvey} for an introduction to the method used by them), by Nizhnik and his collaborators (see  \cite{nizhnik1988}),  
and others; please see \cite{MR2575362} for a brief survey.
Inverse problems for hyperbolic PDEs with multiple speeds of 
propagation present challenges because of the presence of precursor waves. If the inital wave is an impulsive wave travelling 
with a slower speed, then an interaction with the medium (coefficients) may result in a smoother wave moving at a faster 
speed which reaches points in the medium before the more singular initial wave reaches there - this is the precursor wave. Since
techniques used for inverse problems for single speed problems rely on the most singular wave arriving first or at the same 
time as the slower wave, new techniques need to be developed to solve the slower impulsive wave inverse problem. In \cite{belishev1997two}, Belishev et al made an important 
observation and showed the way for solving inverse problems for multi-speed hyperbolic systems, which we describe next.

Define the diagonal matrix $D=\begin{bmatrix} 1 & 0 \\ 0 & c^2 \end{bmatrix}$ with $0<c<1$ and let $P(z), Q(z)$ be 
arbitrary $2\times 2$ matrices. For arbitrary $f_1(t), f_2(t)$, let ${\bf v}(z,t)\in {\mathbb R}^2$ be the solution of the two speed IBVP
\begin{subequations}
\begin{align}
{\bf v}_{tt}- D {\bf v}_{zz}-P{\bf v}_z-Q{\bf v}&=0, \quad  (z,t)\in [0,\infty)\times {\mathbb R}, \label{IBVPt}\\
{\bf v}&=0, \quad  t<0, \\
{\bf v}(0,t)&=[f_1(t), ~f_2(t)]^T, \quad t\in {\mathbb R}. \label{IBVPb}
\end{align}
\end{subequations}
If $f_1, f_2$ are supported in the region $t\ge0$ then, because of the finite speed of propagation, ${\bf v}(z,t)$ is supported in 
the fast region $t \ge z$. 
In \cite{belishev1997two}, Belishev et al showed that there exists a unique $l(\cdot)$ such that ${\bf v}$ is supported in the 
slow region $ct\ge  z$ if $f_2=l*f_1$, where $*$ represents convolution  - see Figure \ref{fig:twospeed}. 
\begin{figure}[!h]
\centering
\epsfig{file=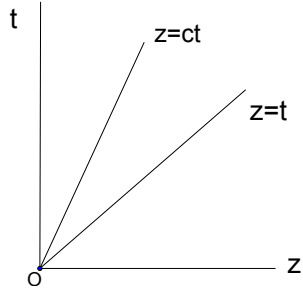, height=1.5in}
\caption{Fast and slow regions}
\label{fig:twospeed}
\end{figure}
Then, in \cite{belishev1997two}, \cite{BI02}, they considered the following inverse problem for a two speed 
hyperbolic system. Let $U(z,t)$ be the $2 \times 2$ matrix 
solution of the impulsive IBVP
\begin{subequations}
\begin{align}
U_{tt}-DU_{zz}-PU_z-QU&=0, \quad (z,t)\in [0,\infty)\times {\mathbb R}, \label{BIt}\\
U&=0, \quad  t<0, \\
U(0,t)&=\delta(t)I_2, \quad t\in {\mathbb R}; \label{BIb}
\end{align}
\end{subequations}
here $I_2$ is the $2\times 2$ identity matrix.
Their goal was the recovery of the coefficients of $P(z)$ and $Q(z)$ over some interval $[0,Z]$, given $U_z(0,t)$ for all $t$ in
some interval $[0,T]$. The problem as stated is under-determined and is under-determined even if we assume the differential 
operator is self-adjoint - that is if the diagonal entries of $P$ are zero and $Q-Q^T=P'$. Belishev et al showed that, in the self-
adjoint case, one can recover $P(z)$ and $Q(z)$ if one is given $l(\cdot)$ in addition to $U_z(0,\cdot)$. They also had a 
data characterization result in \cite{BI02} which is summarized in the introduction of \cite{MR2575362}.

For our problem, the goal is inversion without knowledge of $l(\cdot)$. In this direction, in \cite{BI03}, Belishev et al showed 
that if only $U_z(0, \cdot)$ is given (and $l(\cdot)$ is not given) and some of the coefficients of $P(z), Q(z)$ are known then
$l(t)$ can be recovered over a small interval $[0, \delta]$ and hence the remaining coefficients of $P(z)$ and $Q(z)$ could be
recovered over a small interval. This result was used by Morassi et al in \cite{mns05} to prove a uniqueness result. Please see 
the introduction to \cite{MR2575362} for a summary of these results. The recovery of $l(\cdot)$ over the full interval is an open 
question. 

The article \cite{MR2575362} also studies the recovery of $P,Q$ from $U_z(0,\cdot)$ without knowledge of $l(t)$; a 
stability result is proved if some of the coefficients of $P,Q$ are known but no reconstruction is provided. Please see the article 
for a careful statement. 

Our work focuses on the reconstruction of a single coefficient of a two speed hyperbolic system without the knowledge of $l(t)$. 
We are given less data but we have to recover only one coefficient $\beta(z)$. We have borrowed ideas for inverse 
problems for single speed hyperbolic problems in \cite{MR706374}, \cite{MR860922}. Normally this would fail for
 two speed
problems for reasons pointed out above but due to the special structure of our problem we have succeeded in applying single 
speed 
ideas to our problem and proved Theorems \ref{stab} and \ref{thm:recon}.

This article is partly based on some of the work in the PhD thesis of Jiahua Tang.                         
%
%
\section{Proof of Theorem \ref{wellposed}}\label{sec:forward}

We show that the IBVP (\ref{IVP11})-(\ref{IVP21}) is well posed. The solution $\M(z,t)$ is a distribution 
and it will be 
useful to express it in terms of standard 
distributions and well behaved functions. Using the progressing wave expansion and proceeding in a fashion similar to the 
derivation of Theorem 3 in \cite{MR2575362}, we can show that (see Figure \ref{geobc2}),
\begin{align}
{\bf M}(z,t)=\delta(t-z)[0,1,0,0]^T+{\bf q}(z,t)H(t-z/c)+{\bf p}(z,t)(H(t-z)-H(t-z/c)) \label{Mc<1}
\end{align}
where $\p(z,t), \q(z,t)$  is the solution of the characteristic transmission BVP 
\begin{subequations}
\begin{align}
{\cal L} {\bf p}&={\bf 0} \quad \text{on} ~ 0\le ct\le z\le t, \label{IBVPc<1t} \\
{\cal L} {\bf q}&={\bf 0} \quad \text{on} ~ 0\le z\le ct,\\
p_1(z,z)=0,\quad p_3(z,z)&=\frac{c-1}{2(1+c)}\beta(z),\quad p_4(z,z)=\frac{1+c}{2(c-1)}\beta(z),\quad z\ge 0,\\
(q_1-p_1)(z,t)=(q_2-p_2)(z,t)&=(q_3-p_3)(z,t)=0 \quad \text{on} ~ z=ct,~z\ge 0, \\
q_2(0,t)&=q_4(0,t)=0,\quad t\ge 0. \label{IBVPc<1b}
\end{align}
\end{subequations}
\begin{figure}
\centering
\epsfig{file=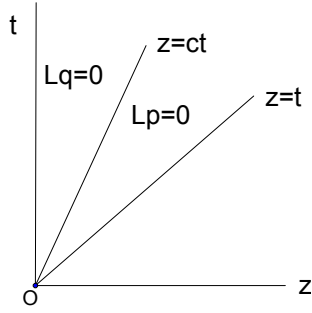, height=1.6in}
\caption{Subregions}
\label{geobc2}
\end{figure}

For $Z>0$, define (see Figure \ref{d1d2_3})
\begin{align*}
D_1 &:=\{(z,t)~ |~ ct\le z\le t, z+t\le 2Z\}, \\
D_2 &:=\{(z,t)~ |~  0\le z\le ct , z+t\le 2Z\},\\
D &:=D_1\cup D_2.
\end{align*}
\begin{figure}[h]
\centering
\epsfig{file=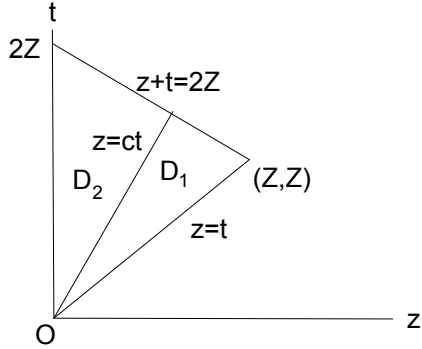, height=1.8in}
\caption{$D_1$ and $D_2$}
\label{d1d2_3}
\end{figure}

The well-posedness of  (\ref{IBVPc<1t})-(\ref{IBVPc<1b})  will follow from the well-posedness
of the following general characteristic transmission BVP.
\begin{subequations}
\begin{align}
{\cal L} {\bf f}&={\bf 0}~ \text{in}~ D_1, \label{moregent}\\
{\cal L} {\bf g}&={\bf 0}~ \text{in} ~ D_2,\label{mg2}\\
(g_i-f_i)(ct,t)&= 0,\quad i = 1,2,3,\quad t\in \left[0, \frac{2Z}{1+c}\right], \label{mg3} \\
f_1(t,t)=b_1(t),\quad f_3(t,t)&=b_3(t),\quad f_4(t,t)=b_4(t),\quad t\in [0,Z], \\
g_2(0,t)=e_2(t),\quad g_4(0,t)&=e_4(t),\quad t\in [0,2Z]. \label{moregenb}
\end{align}
\end{subequations}
One may verify that if $\f \cup \g$ is in $C^1(D)$ and $\beta(0)=0$ then
\beqn
 b_4(0)=e_4(0), ~~ 2(1-c)e_4'(0)=(1-c)b_1(0)-  2cb_4'(0)  - (1+c)e_2(0).
\label{matching}
\eeqn
We have the following result regarding the well-posedness of (\ref{moregent})-(\ref{moregenb}).

\begin{proposition}[Existence of $C^1$ solutions] \label{moregen1}
If $\beta(\cdot) \in \dot{C}^1[0,2Z]$,  $b_i(t) \in C^1[0,Z]$, and $e_i(t)\in C^1[0,2Z]$ and satisfy (\ref{matching})
then there exists a unique solution ${\bf f}\in C^1(D_1), {\bf g}\in C^1(D_2)$ of $(\ref{moregent})-(\ref{moregenb})$
with $\|{\bf f}\|_{C^1}$, $\|{\bf g}\|_{C^1}$  bounded above by a function of $c,Z$
and
$ N=\max ( \|\beta\|_{C^0}, \|b_i\|_{C^1}, \|e_i\|_{C^1} ).$
Further $\f \cup \g$ is a $C^1$ function on $D$.
\end{proposition}


\begin{proof}[Proof of Proposition \ref{moregen1}]
The existence of the solution will be reduced to the solution of an integral equation. Below $\v(z,t)$ will 
represent a 4 component vector function on $D_1 \cup D_2$ and ${\bf r}({\bf v},z,t) =\beta(z) B{\bf v}(z,t)$.

By integrating (\ref{moregent})-(\ref{mg2}) along the characteristics and using the boundary conditions, we
may show that the existence of a classical solution of (\ref{moregent})-(\ref{mg2}) reduces to solving
the system of integral equations (see Figure \ref{PD12})
\begin{subequations}
\begin{align}
v_1(z,t)&=
      \int_{s_H}^tr_1({\bf v},z+t-s,s)~ds+b_1(s_H), \qquad~ \text{if} ~ P \in D
 \label{v1} \\
v_2(z,t)&=
       \int_{s_E}^tr_2({\bf v},z+s-t,s)~ds+e_2(s_E), \qquad ~ \text{if}~ P \in D
 \label{v2}\\
v_3(z,t)&=
       \int_{s_G}^tr_3({\bf v},z+ct-cs,s)~ds+b_3(s_G), \qquad~ \text{if}~ P \in D
 \label{v3} \\
v_4(z,t)&=\left\{
     \begin{array}{lr}
       \int_{s_F}^tr_4({\bf v},z+cs-ct,s)~ds+b_4(s_F), \qquad ~ \text{if} ~ P \in D_1\\
       \int_{s_F}^tr_4({\bf v},z+cs-ct,s)~ds+e_4(s_F), \qquad~ \text{if}~ P \in D_2
     \end{array}
\right.
\label{v4}
\end{align}
\end{subequations}
where $s_E, s_F, S_G, S_H$ are the $s$ coordinates of the points $E,F, G, H$ in the $(y,s)$ plane in Figure \ref{PD12}.
\begin{figure}[h]
\centering
\subfigure{
\epsfig{file=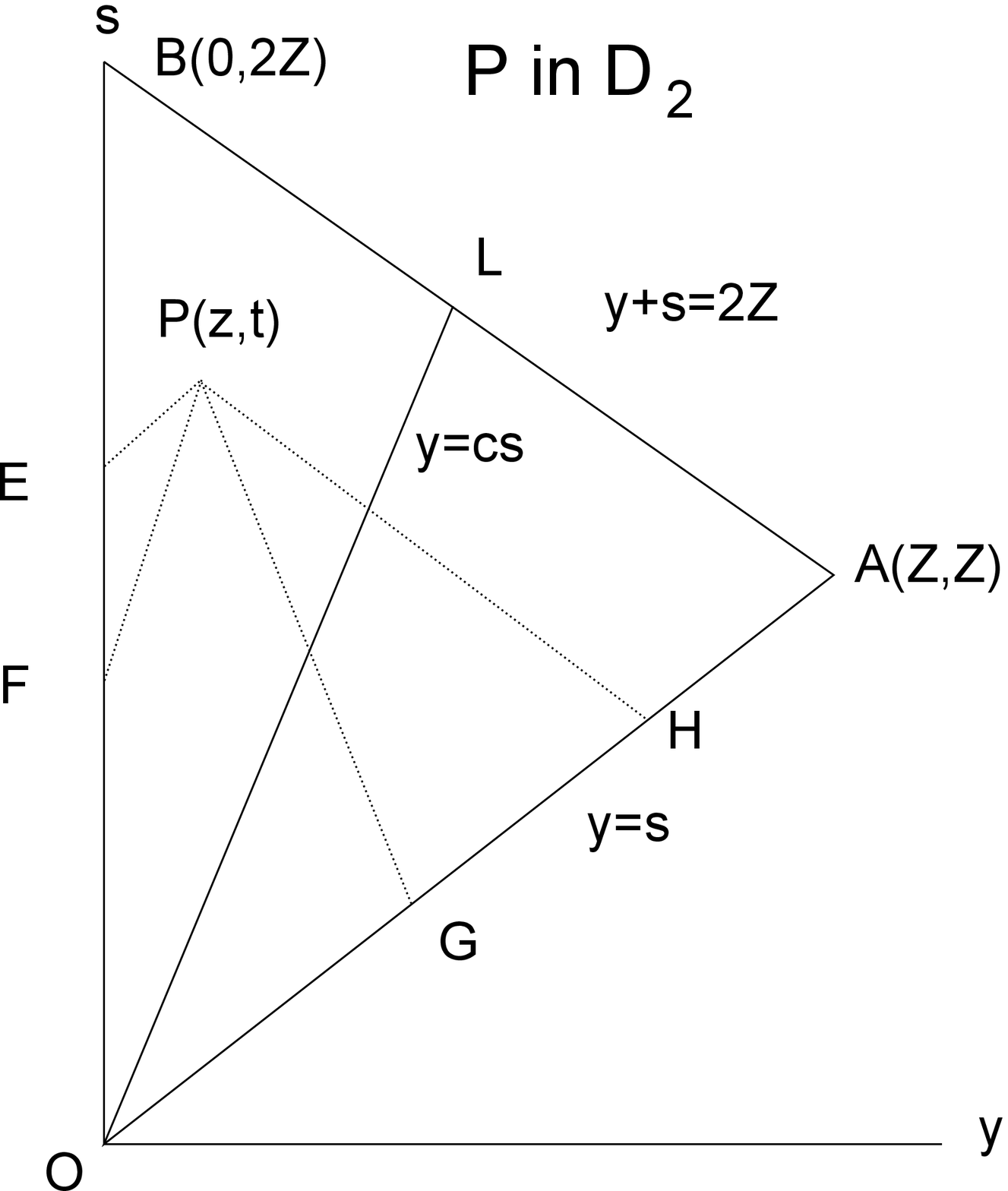, height=2.4in}
}
\hspace{0.3in}
\subfigure{
\epsfig{file=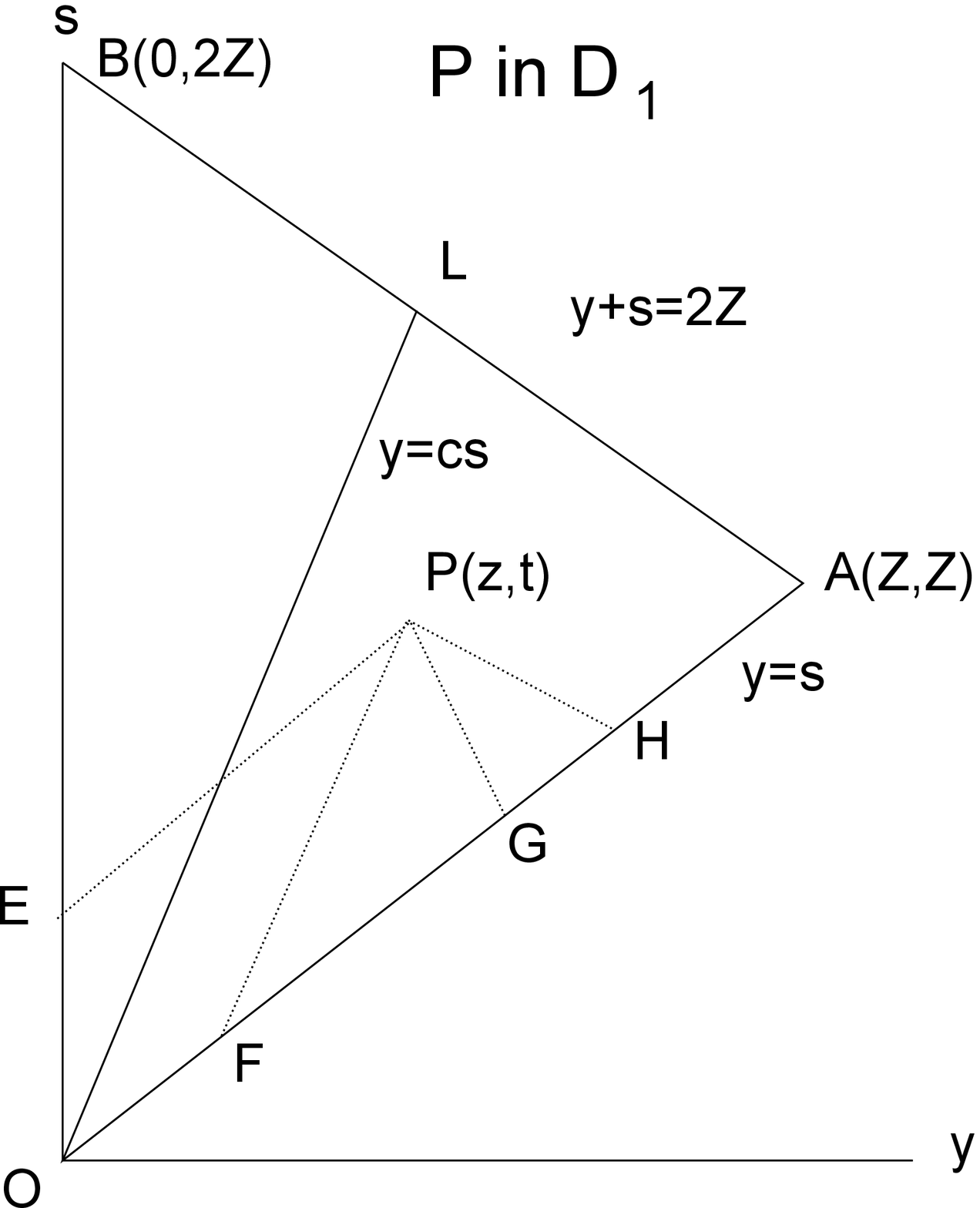, height=2.4in}
}
\caption{Downward moving lines through $P(z,t)$ with slopes $\pm 1$ and $\pm 1/c$}
\label{PD12}
\end{figure}

If $\v$ is in $C^1(D)$ then one may verify that the RHS of 
(\ref{v1})-(\ref{v4}) is in $C^1(D)$ - the first order derivatives on $z=ct$ match as one approaches this line from the two 
different sides. Further, the system of integral equations is a Volterra type equation so the 
existence and uniqueness of a $C^1$ solution may be proved by standard arguments for Volterra equations or one may use the 
method in section 2.5 of \cite{MR1185075}.
\end{proof}



\section{Proof of Theorem \ref{stab}}

We will use the following identity in several places in this article: for arbitrary four dimensional $C^1$ vector functions
 ${\bf u}(z,t), {\bf v}(z,t)$,   since $B^T=-B$, we can verify that
\begin{align}
{\bf u}^T({\cal L} {\bf v})+({\cal L} {\bf u})^T{\bf v}=({\bf u}^T{\bf v})_t-({\bf u}^TA{\bf v})_z.
\label{eq:ulv}
\end{align}
\begin{figure}[h]
\centering
\epsfig{file=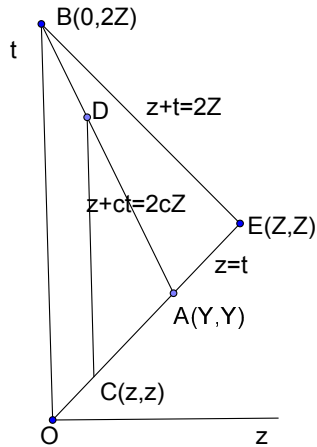, height=2.3in}
\caption{Regions used for the energy function}
\label{CDest}
\end{figure}
The proofs of Theorem \ref{stab} and Theorem \ref{thm:recon} will use sideways energy estimates given below.
Define (the triangle OAB in Figure \ref{CDest})
\begin{align*}
{\tilde D} :=\{(z,t)~|~0\le z\le Y, z\le t\le 2Z-z/c\}
\end{align*}
and for any 
4 dimensional vector function $\p(z,t) \in C^1(\D)$, $\ep>0$ define 
\begin{align*}
\text{(sideways energy)} \qquad J({\bf p},z) :=\int_{CD}(p_1^2+p_2^2+cp_3^2+c\epsilon p_4^2)(z,t)~dt,\quad z\in [0,Y],
\\
|{\bf p}(z,t)|^2 := \sum_{i=1}^4p_i^2(z,t),
\qquad
|({\cal L} {\bf p})(z,t)|^2 := \sum_{i=1}^4| ({\cal L} \p)_i|^2(z,t),\quad (z,t)\in {\tilde D}.
\end{align*}

\begin{lemma}[Sideways energy estimate]\label{lemma:energy}
If $\beta \in L^2[0,Y]$,  ${\bf p}(z,t)\in C^1({\tilde D})$, then for every $\lambda >0$, $\epsilon\in (0,1]$ and $z\in [0,Y]$ 
$\left(\text{see Figure \ref{CDest}} \right)$ 
\begin{align}
J({\bf p},z)&+\int_{OC}(2p_1^2+(1+c)p_3^2-\epsilon(1-c)p_4^2)(y,y) ~dy 
\nn\\
\le & J({\bf p},0)+\lambda \iint_{OCDB}|({\cal L} {\bf p})(y,t)|^2~dx \, dt + \frac{1}{c\epsilon}\int_0^z\left(4|\beta(y)|+\frac{1}{\lambda}\right)J({\bf p},y)~dy.  \label{LRHSO}
\end{align}
Furthermore, if $\epsilon \le \frac{c(1-c)^3}{(1+c)^4}$ and ${\bf p}$ satisfies
\begin{subequations}
\begin{align}
{\cal L} {\bf p}&=0,\quad \text{in}~{\tilde D},\label{poundt}\\
p_3(z,z)&=\frac{(c-1)^2}{(c+1)^2} \, p_4(z,z),\quad z\in [0,Y],\label{poundb}
\end{align}
\end{subequations}
then 
\begin{align}
J({\bf p},z) +  \int_{OC} p_3^2 \, dy
 \le e^{4\sqrt{Y \|\beta\|}/(c\epsilon)}J({\bf p},0),\quad z\in [0,Y]    \label{jmz04}
\end{align}
where $\|\beta\|$ is the $L^2$ norm of $\beta$ on $[0,Y]$.
\end{lemma}
%
%
We postpone the proof of Lemma \ref{lemma:energy} to subsection \ref{subsec:lemma} and continue with the proof of 
Theorem \ref{stab}.

Let ${\bf m}$, ${\bf \tilde m}$ be the solutions of $(\ref{compt})-(\ref{compb})$ corresponding to $\beta$, ${\tilde 
\beta}$ in $\B_K$. Since $\tilde{\L} \mtb =0$ and $\mtb^T B \mtb=0$ because $B^T=-B$, we have
\begin{align*}
0 &=  2\mtb^T \tilde{\L} \mtb  = 2\mtb^T \mtb_t - 2 \mtb^T A \mtb_z -2 \tilde{\beta} \, \mtb^T B \mtb = 
2\mtb^T \mtb_t - 2\mtb^T A \mtb_z
\\
& = (\mt_1^2 + \mt_2^2 + \mt_3^2 + \mt_4^2)_t - (\mt_1^2 - \mt_2^2 + c\mt_3^2 -c \mt_4^2)_z.
\end{align*}
Integrating this over the triangular region OEB (see Figure \ref{CDest}) and noting (\ref{eq:m1bc}), we obtain
\begin{align*}
 \int_{OE} 2\mt_1^2 + (1+c) \mt_3^2 + (1-c) \mt_4^2 \, dz
=  \int_{EB} 2 \mt_2^2 + (1-c) \mt_3^2 + (1+c) \mt_4^2 \, dt
+ \int_{OB} \mt_1^2 + c \mt_3^2 \, dt.
\end{align*}
Hence, using (\ref{compb}), we have
\beqn
J( \mtb, 0) \leq  \int_{OE} 2\mt_1^2 + (1+c) \mt_3^2 + (1-c) \mt_4^2 \, dz
= \frac{1+3c^2}{2 (1-c^2)} \int_0^Z \betat(z)^2 \, dz \leq  \frac{1+3c^2}{2 (1-c^2)} \, K.
\label{eq:Jm0}
\eeqn
Next, applying Lemma \ref{lemma:energy} with $\p$ replaced by $\mtb$, $\beta$ replaced by $\betat$
and $\ep = c (1-c)^3/(1+c)^4$ - note that 
(\ref{poundt}), (\ref{poundb}) hold - from (\ref{jmz04}) and (\ref{eq:Jm0}) we obtain
\beqn
J( \mtb, z) \leq \frac{1+3c^2}{2 (1-c^2)} \,  K \,  e^{4\sqrt{Y} K^{1/4}/(c\epsilon)} = C_0 ~ (\text{define}).
\label{eq:Jmzest}
\eeqn
 
Define ${\bf p}={\bf m}-{\bf \tilde m}$;  then ${\bf p}$ satisfies ${\cal L} {\bf p}=(\beta-{\tilde \beta})B{\bf 
\tilde m}$ and, from (\ref{compb}), we have
\beqn
p_3(z,z)=\frac{(c-1)^2}{(c+1)^2} \,p_4(z,z) = \frac{ c-1}{2 (c+1)} ( \beta - \tilde{\beta})(z,z), \qquad z \in [0,Y].
\label{eq:pmmt}
\eeqn
Choose $\epsilon=\frac{c(1-c)^3}{(1+c)^4}$, then
\begin{align*}
\int_{OC}((1+c)p_3^2-\epsilon(1-c)p_4^2)~dy=
\int_{OC} p_3^2 \, dy = \frac{(1-c)^2}{4(1+c)^2}\int_{OC}(\beta-{\tilde \beta})^2(y)~dy,
\end{align*}
so from $(\ref{LRHSO})$ in Lemma \ref{lemma:energy} and $(\ref{eq:Jmzest})$ we have
\begin{align}
J({\bf p},z)&+\frac{(1-c)^2}{4(1+c)^2}\int_0^z(\beta-{\tilde \beta})^2(y)~dy
\nn\\
& \le J({\bf p},0)+4 \lambda \int_0^z (\beta-{\tilde \beta})^2(y) \, J(\mtb,y) ~dy + 
\frac{1}{c\epsilon}\int_0^z\left(4|\beta(z)|+\frac{1}{\lambda}\right)J({\bf p},y)~dy
 \nn \\
& \le J({\bf p},0)+4C_0\lambda \int_0^z(\beta-{\tilde \beta})^2(y)~dy + 
\frac{1}{c\epsilon}\int_0^z\left(4|\beta(z)|+\frac{1}{\lambda}\right)J({\bf p},y)~dy.
 \label{cheese4}
\end{align}
Choose $\lambda=\frac{(1-c)^2}{32C_0(1+c)^2}$, then from $(\ref{cheese4})$ we have
\begin{align}
J({\bf p},z)+\frac{(1-c)^2}{8(1+c)^2} \int_0^z(\beta-{\tilde \beta})^2(y)~dy
\le J({\bf p},0)+\frac{1}{c\epsilon}\int_0^z\left(4|\beta(z)|+\frac{1}{\lambda}\right)J({\bf p},y)~dy;
 \label{jpz4}
\end{align}
hence from Gronwall's inequality
\begin{align*}
\int_0^Y(\beta-{\tilde \beta})^2(y)~dy\preceq J({\bf p},0)\le \int_0^{2T}((m_1-{\tilde m}_1)^2+(m_3-{\tilde m}_3)^2)(0,t)~dt,
\end{align*}
with the constant dependent only on $c,Z,K$.
This completes the proof of Theorem \ref{stab}.
%


\subsection{Proof of Lemma \ref{lemma:energy}}\label{subsec:lemma}

Define ${\bf q} :=[p_1,~-p_2,~p_3,~-\epsilon p_4]^T$; multiplying both sides of ${\cal L} {\bf p}={\bf p}_t-A{\bf p}_z-\beta B{\bf p}$ by $-2{\bf q}^T$, we have
\begin{align}
-2{\bf q}^T({\cal L} {\bf p}+\beta B{\bf p})=&2{\bf q}^T(A{\bf p}_z-{\bf p}_t)\nn \\
=&(p_1^2+p_2^2+cp_3^2+\epsilon cp_4^2)_z-(p_1^2-p_2^2+p_3^2-\epsilon p_4^2)_t. \label{LRHS}
\end{align}
Integrating the RHS of $(\ref{LRHS})$ over $OCDB$, we have
\begin{align}
&\iint_{OCDB}(p_1^2+p_2^2+cp_3^2+c\epsilon p_4^2)_z-(p_1^2-p_2^2+p_3^2-\epsilon p_4^2)_t~dx \, dt\nn\\
=&\int_{\partial OCDB}(p_1^2-p_2^2+p_3^2-\epsilon p_4^2)~dz+\int_{\partial OCDB}(p_1^2+p_2^2+cp_3^2+c\epsilon p_4^2)~dt\nn\\
=&\int_{OC}(p_1^2+p_2^2+cp_3^2+c\epsilon p_4^2)~dt+J({\bf p},z)-J({\bf p},0)+\int_{DB}(p_1^2+p_2^2+cp_3^2+c\epsilon p_4^2)~dt\nn\\
&-\int_{DB}c(p_1^2-p_2^2+p_3^2-\epsilon p_4^2)~dt+\int_{OC}(p_1^2-p_2^2+p_3^2-\epsilon p_4^2)~dt\nn\\
=&\int_{DB}((1-c)p_1^2+(1+c)p_2^2+2c\epsilon p_4^2)~dt+\int_{OC}(2p_1^2+(1+c)p_3^2-\epsilon(1-c)p_4^2)~dt\nn\\
&+J({\bf p},z)-J({\bf p},0). \label{LHSO}
\end{align}
Also
\begin{align}
\iint_{OCDB}&\text{LHS of (\ref{LRHS})}~dx \, dt\nn\\
&\le \iint_{OCDB}\left(\frac{|{\bf p}(y,t)|^2}{\lambda}+\lambda|({\cal L} {\bf p})(y,t)|^2+4|\beta(y)|\cdot |{\bf p}(y,t)|^2\right)~dx \, dt\nn\\
&\le \lambda\iint_{OCDB}|({\cal L} {\bf p})(y,t)|^2~dx \, dt 
+\frac{1}{c\epsilon}\int_0^z\left(4|\beta(y)|+\frac{1}{\lambda}\right)J({\bf p},y)~dy
 \label{RHSO}
\end{align}
for all $\lambda >0$ and $\epsilon\in (0,1]$, so $(\ref{LRHSO})$ follows directly from $(\ref{LHSO})-(\ref{RHSO})$.\\\\
If ${\bf p}$ satisfies $(\ref{poundt})-(\ref{poundb})$ and $\epsilon \le \frac{c(1-c)^3}{(1+c)^4}$, then
\begin{align*}
\int_{OC}(2p_1^2+(1+c)p_3^2-\epsilon(1-c)p_4^2)~dy\ge  \int_{OC} p_3^2 \, dy.
\end{align*}
In $(\ref{LRHSO})$ using $(\ref{poundt})$ and letting $\lambda\to \infty$, we have
\begin{align*}
J({\bf p},z) + \int_{OC} p_3^2 \, dy
\le J({\bf p},0)+\frac{4}{c\epsilon}\int_0^z|\beta(y)|J({\bf p},y)~dy,
\end{align*}
so $(\ref{jmz04})$ follows directly from Gronwall's inequality.
%


\section{Proof of Theorem \ref{thm:recon}}

For arbitrary $Z>0$, $K>0$, define $Y=\dfrac{2cZ}{1+c}$ and (see Figure \ref{OCAB})
\begin{align*}
\B_K &:=\{\beta \in \dot{C}^1[0,Z]~ |~ \|\beta(\cdot)\|_{L^2[0,Z]}^2\le K\},
\\
\Omega &:=OAB=\{(z,t)~|~ z\ge 0, ~z\le t, ~z+t\le 2Z\}, 
\\
\Om &:=OCB=\{(z,t)~|~0\le z\le Y, z\le t\le 2Z-z/c\}.
\end{align*}
\begin{figure}[!h] 
\centering
\epsfig{file=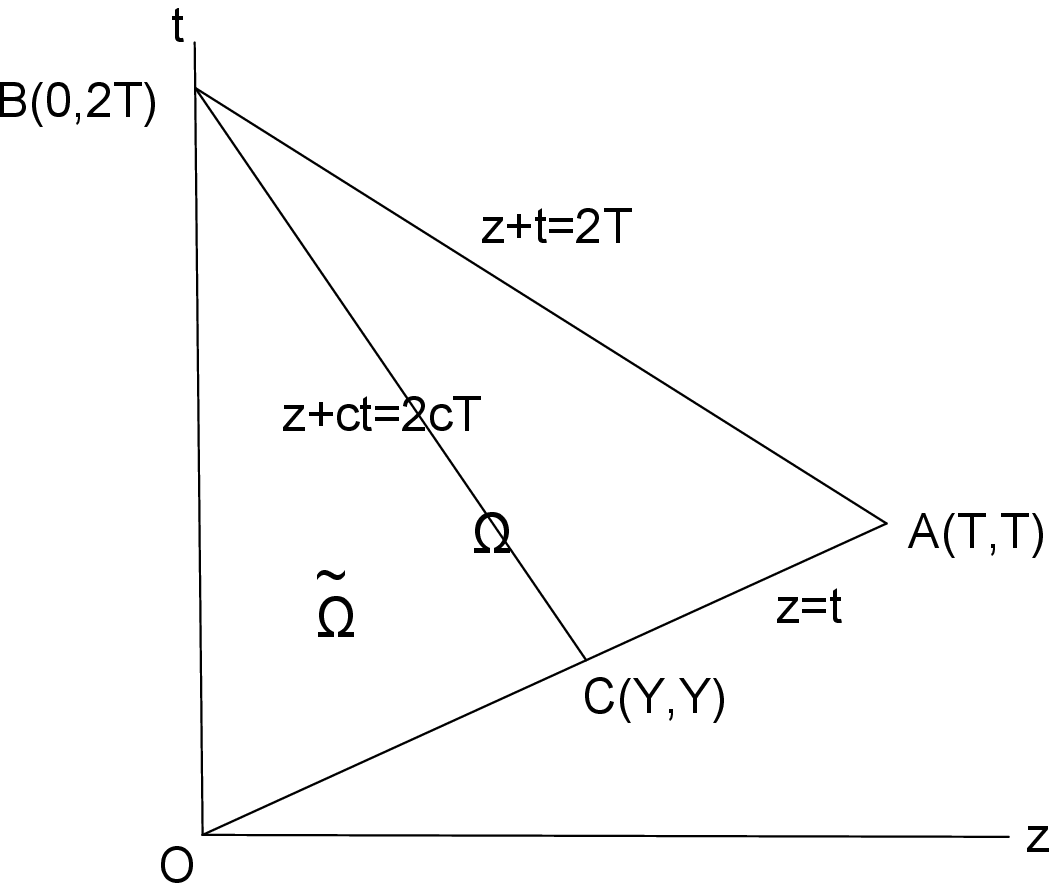, height=2in}
\caption{$\Omega$ and ${\Om}$} 
\label{OCAB}
\end{figure}
\noindent
Because of Theorem \ref{wellposed}, our goal is to recover $\beta(\cdot)$ on $[0,Y]$ given $m_1(0,\cdot)$, $m_3(0,\cdot)$ 
on $[0,2Z]$ where $\m(z,t)$ is the unique $C^1$ solution of (\ref{compt})-(\ref{compb}).
So our goal is the construction of the partial\footnote{ because we recover $\beta(\cdot)$ only on $[0,Y]$} inverse (on the range) of the injective nonlinear map
\begin{align*}
\F : \dot{C}^1[0,Z] & \to C^1[0,2Z] \times C^1[0,2Z]
\\
\beta( \cdot) & \to ( m_1(0, \cdot), m_3(0, \cdot) ).
\end{align*}

Fix an $(a_1(\cdot), a_3(\cdot))$ in the range of $\F$. For any $\beta(\cdot) \in \dot{C}^1[0,Y]$, consider the 
sideways problem
\begin{subequations}
\begin{align}
{\cal L} {\bf h}&={\bf 0} ~ \text{in} ~ {\Om}, \label{eq:hmsde} \\
h_1(0,t)=a_1(t),\quad h_2(0,t)&=0, \quad h_3(0,t)=a_3(t),\quad h_4(0,t)=0 ,\quad 0\le t\le 2Z,
\label{eq:hmsbc}
\\
h_3(z,z)&=\frac{(c-1)^2}{(c+1)^2}h_4(z,z),\quad 0\le z\le Y. 
\label{eq:hmscc}
\end{align}
\end{subequations}
Note that if $\h$ satisfies (\ref{compb}) then $\h$ satisfies (\ref{eq:hmscc}).
We show that (\ref{eq:hmsde})-(\ref{eq:hmscc}) has a unique $C^1$ solution for every 
$\beta (\cdot)  \in \dot{C}^1[0,Y]$. But more importantly, we then show, {\bf constructively}, that there is a unique 
$\beta (\cdot)  \in \dot{C}^1[0,Y]$ such that
\[
h_3(z,z) = \frac{ (c-1)}{2 (c+1)} \beta(z), \qquad 0 \leq z \leq Y.
\]
But we already know one such $\beta$. Since $(a_1(\cdot), a_3(\cdot))$ is in the range of $\F$, there is a $\beta$ and an
$\m$ which solves (\ref{compt})-(\ref{compb}) and such that $\F(\beta) = ( a_1 (\cdot), a_2(\cdot) )$. Since this $\m$ will also 
satisfy (\ref{eq:hmsde})-(\ref{eq:hmscc}), the unique $\beta$ found above must be the preimage of
$(a_1(\cdot), a_3(\cdot))$ under $\F$.

Of the two claims mentioned in the previous paragraph, the first one about the well-posedness of the CBVP  
(\ref{eq:hmsde})-(\ref{eq:hmscc}) will follow from a standard argument. The second claim, about a nonlinear 
problem, 
will be shown to be equivalent to the solution of a fixed point problem which will be studied by a contraction mapping 
argument.
This will 
take some work because we will have to extend the idea of a solution of (\ref{eq:hmsde})-(\ref{eq:hmscc}) to the case 
where $\beta \in L^2[0,Y]$ because our estimates will be $L^2$ estimates and hence to apply the contraction mapping theorem we will have to work square integrable $\beta$.


\subsection{Well-posedness for the sideways CBVP}

For this subsection, we drop the assumption that $\beta(0)=0, \beta'(0)=0$, for reasons which will become clear in the next subsection. For an arbitrary 4 dimensional vector function $\a(t) \in C^1[0,2Z]$, consider the CBVP
\begin{subequations}
\begin{align}
{\cal L} {\bf h}&={\bf 0} ~ \text{in} ~ {\Om}, \label{sidetc4}\\
h_3(z,z)&=\frac{(c-1)^2}{(c+1)^2}h_4(z,z),\quad 0\le z\le Y, \label{sidem}\\
\h(0,t) &=\a(t), \quad 0\le t\le 2Z.\label{sidebc4}
\end{align}
\end{subequations}
A simple but tedious calculation shows that the necessary (matching) conditions for (\ref{sidetc4})-(\ref{sidebc4}) to have a 
$C^1$ solution are that
\begin{subequations}
\begin{align}
(c+1)^2a_3(0) &=(1-c)^2a_4(0),\label{compat44}\\
(c+1)^2\left((1+c)a_3'(0)- \beta(0) (B \a)_3(0) )\right)
&=(c-1)^2\left((c-1)a_4'(0)+\beta(0) ( B \a)_4(0) )\right). \label{compat4}
\end{align}
\end{subequations}

\begin{proposition}[Well-posedness of the sideways CBVP] \label{prop:sideways}
If $\beta \in C^1[0,Y]$, and $\a(\cdot) \in C^1[0,2Z]$ satisfies (\ref{compat44})-(\ref{compat4}) then 
 (\ref{sidetc4})-(\ref{sidebc4}) has a unique solution ${\bf h}\in C^1({\Om})$. 
\end{proposition}
\begin{proof}
Define ${\bf r}({\bf h},z,t) :=\beta(z) B{\bf h}(z,t)$. Integrating $(\ref{sidetc4})$ along the characteristics and using the 
boundary conditions $(\ref{sidem})-(\ref{sidebc4})$, it is clear that proving Proposition \ref{prop:sideways} is equivalent to 
proving 
that the following system of integral equations has a unique $C^1$ solution (see Figure \ref{i12}); here $
P(z,t)$ is an arbitrary point in $\Om$.
\begin{subequations}
\begin{align}
h_1(z,t)&=\int_0^zr_1({\bf h},y,z+t-y)~dy+a_1(s_C), \label{hh14} \\
h_2(z,t)&=\int_0^zr_2({\bf h},y,y+t-z)~dy+a_2(s_D),\\
h_3(z,t)&=\frac{1}{c} \int_0^zr_3\left({\bf h},y,\frac{z+ct-y}{c}\right)~dy+a_3(s_E), \label{v34} \\
h_4(z,t)&=\left\{
     \begin{array}{lr}
       \frac{1}{c}\int_{0}^z r_4 \left( \h,y,\frac{y+ct-z}{c}\right)~dy+a_4(s_F) & \text{if}~z\le ct\\
      \frac{1}{c} \int_{y_H}^zr_4\left(\h,y,\frac{y+ct-z}{c}\right)~dy
          +\frac{(1+c)^2}{(1-c)^2}( \frac{1}{c} \int_0^{y_H}r_3(\h,y,s_F - \frac{y}{c})~dy
          +a_3(s_F))  & \text{if}~z\ge ct.
     \end{array}
   \right.
 \label{hh44}
\end{align}
\end{subequations}
Here $s_E, s_C,y_H, \cdots$ are the $s,y$ coordinates of $E,C,H, \cdots$ in Figure \ref{i12}.
\begin{figure}[!h]
\centering
\subfigure{
\epsfig{file=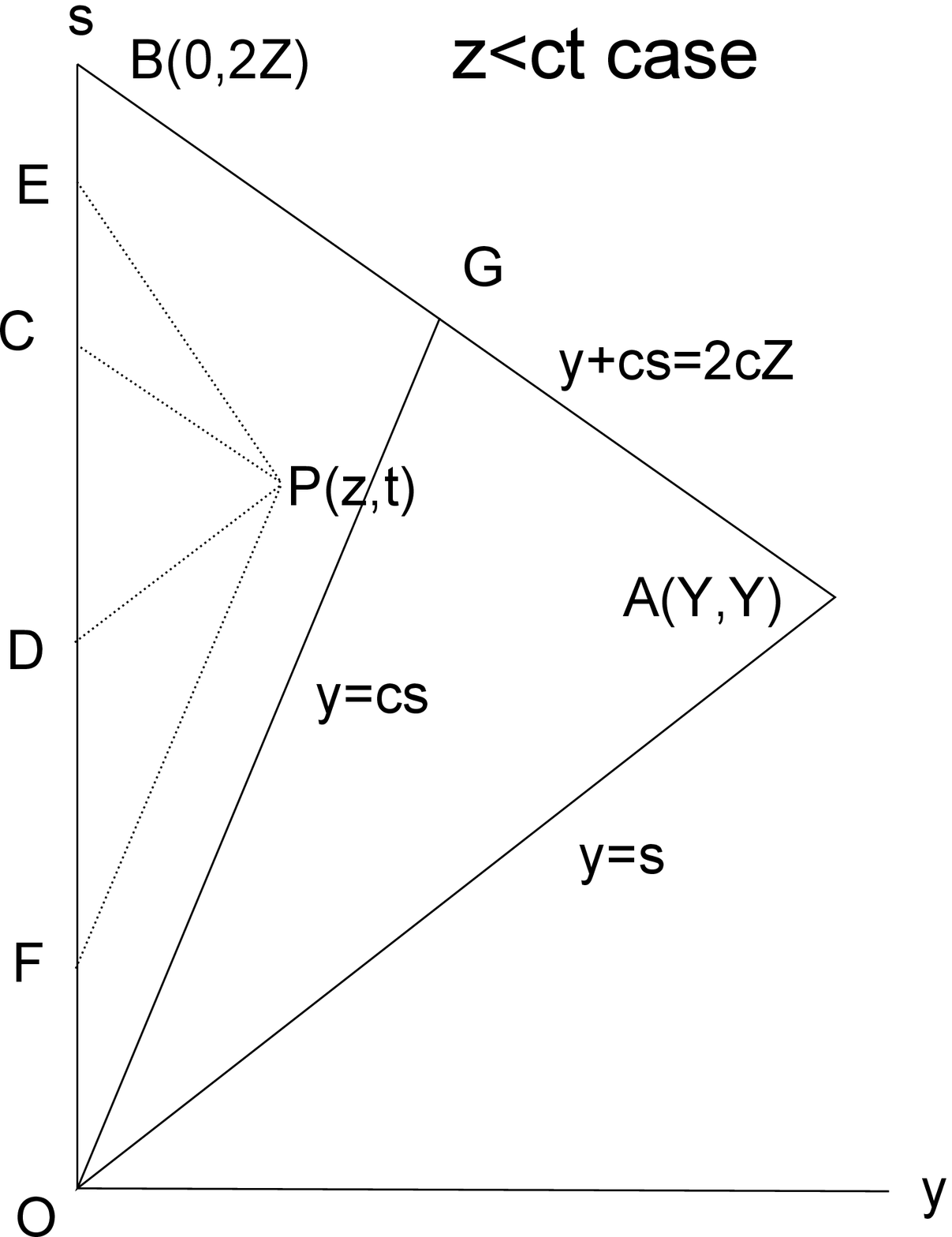, height=2.4in}
}
\hspace{0.2in}
\subfigure{
\epsfig{file=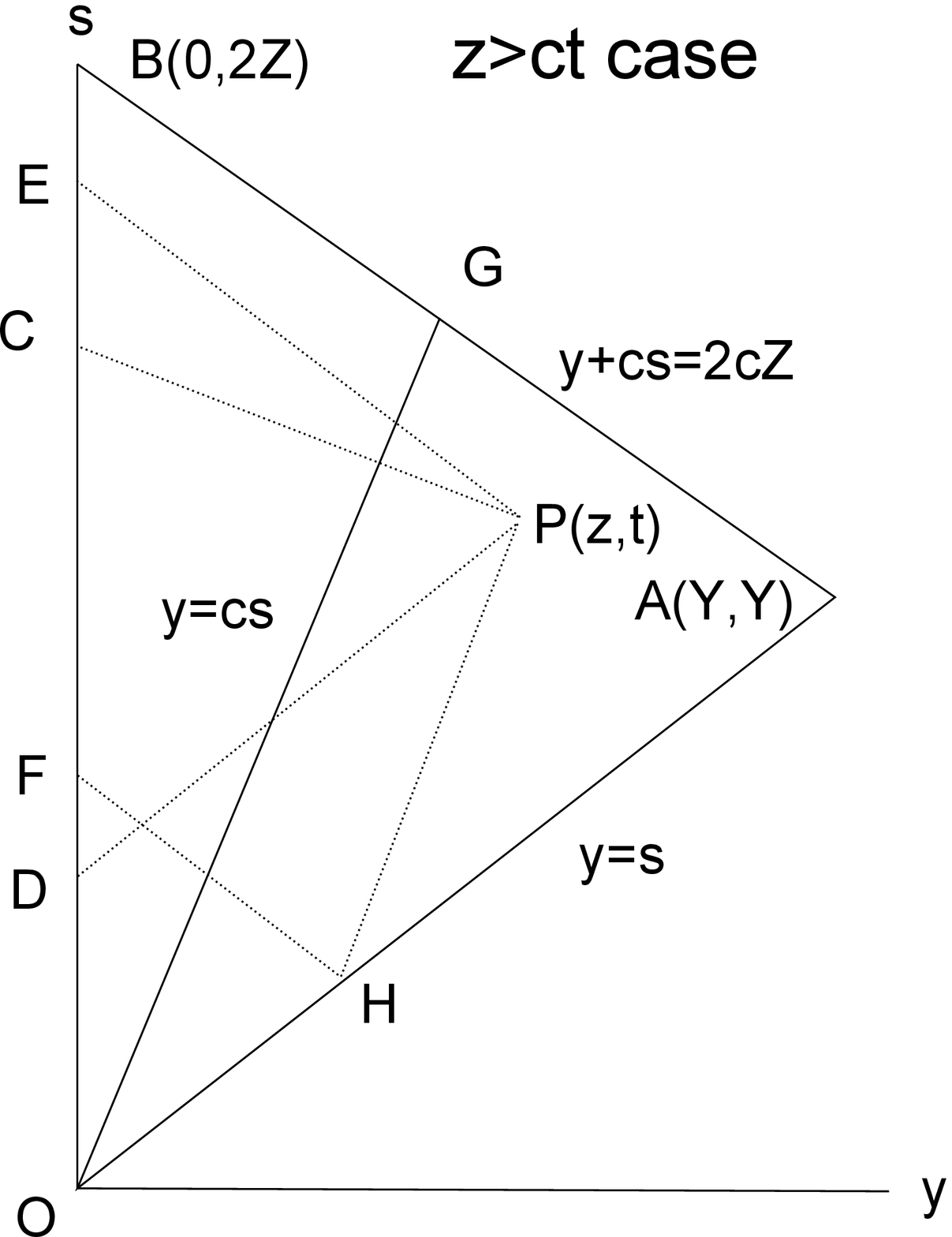, height=2.4in}
}
\caption{Leftward moving lines through $P(z,t)$ with slopes $\pm 1$ and $\pm 1/c$}
\label{i12}
\end{figure}

Again, the existence and uniqueness of a unique $C^1$ solution for (\ref{hh14})-(\ref{hh44}) may be proved by standard arguments 
for Volterra equations or one may use the method in section 2.5 of \cite{MR1185075}.
Of course, because of the piecewise nature of the fourth integral equation (\ref{hh44}), some
calculations are needed to confirm the $C^1$ regularity of $h_4$ and the matching conditions
(\ref{compat44})-(\ref{compat4}) will be required for the $C^1$ regularity at the origin.
\end{proof}


To set up the fixed point problem later in this section, we give meaning to and prove the existence of the unique 
solution  of  (\ref{sidetc4})-(\ref{sidebc4}) when $\beta(\cdot) \in L^2[0,Y]$. Further, we show that this solution has an 
$L^2$ trace on
the line $t=z$.

\begin{proposition}\label{prop:sideways-lipshitz}
If $\a \in C^1[0,2Z]$ satisfies (\ref{compat44}), (\ref{compat4}), and $\h \in C^1(\Om)$ is the 
corresponding solution of  (\ref{sidetc4})-(\ref{sidebc4}),  then the solution map $\S$ and the solution trace map $\T$
\\
\parbox{2in}
{ 
\begin{align*}
\S: C^1[0,Y] & \to C^1(\Om)
\\
\beta( \cdot) & \to \h(\cdot, \cdot)
\end{align*}
}
~~and~~
\parbox{2in}
{
\begin{align*}
\T : C^1[0,Y] & \to C^1[0,Y]
\\
\beta(\cdot) & \to h_3(z,z)
\end{align*}
}
\\
are locally Lipschitz continuous with respect to the $L^2$ norms on the domain and the 
codomains\footnote{The set containing the range}.
\end{proposition}

\begin{proof}

Suppose $\beta, \tilde{\beta} \in C^1[0,Y]$ and $\| \beta\|^2_{L^2[0,Y]} \leq K$,  $\| \tilde{\beta}\|^2_{L^2[0,Y]} \leq K$ 
and WLOG we assume that $\| \a\|_{L^2[0,2Z]} \leq K$.  
Let ${\bf h},{\bf \tilde h}\in C^1({\Om})$ be the solutions of $(\ref{sidetc4})-(\ref{sidebc4})$ corresponding to $\beta,
{\tilde \beta}$.

Applying Lemma \ref{lemma:energy} to $\th$, from (\ref{jmz04}) we have 
\begin{align}
J(\th,z)\le C   J(\th,0) =: C_0, \quad \forall z\in [0,Y], \label{c04}
\end{align}
where $C_0$ depends only on $c,Z,K$. 

If we define $\p = \h - \th$ then ${\bf p}$ satisfies ${\cal L} {\bf p}=(\beta-{\tilde \beta})B{\bf \tilde h}$ with ${\bf p}(0,t)=0$ for $t\in [0,2Z]$ and
\[
p_3(z,z) = \frac{(c-1)^2}{(c+1)^2} \, p_4(z,z),  \quad 0\le z\le Y.
\] 
Hence, applying Lemma \ref{lemma:energy} to $\p$ and taking $\lambda=1$,
$\ep = \frac{c(1-c)^3}{(1+c)^4}$,   from (\ref{LRHSO}) we obtain (see Figure \ref{CDest})
\begin{align*}
J({\bf p},z) +  \int_{OC} p_3^2(y,y) \, dy 
&\le \iint_{OCDB}|({\cal L} {\bf p})(y,t)|^2~dx \, dt
+\frac{1}{c\epsilon}\int_0^z(1+4|\beta(y)|)J({\bf p},y)~dy
\\
& \leq 4 \int_0^z ( \beta - \tilde{\beta})^2(y) \, J(\tilde{\h}, y) \, dy + 
\frac{1}{c\epsilon}\int_0^z(1+4|\beta(y)|)J({\bf p},y)~dy
\\
&  \leq 4 C_0 \int_0^z ( \beta - \tilde{\beta})^2(y) \, dy + 
\frac{1}{c\epsilon}\int_0^z(1+4|\beta(y)|)J({\bf p},y)~dy
\\
\end{align*}
where we used (\ref{c04}) in the last step. Hence, using Gronwall's inequality, we have
\begin{align*}
J({\bf p},z) +  \int_{OC} p_3^2(y,y) \, dy
&\le C \, \int_0^Y (\beta-{\tilde \beta})^2(y) ~dy
\end{align*}
with the constant dependent only on $c,Z,K$. This is enough to prove the proposition.
\end{proof}

Since $C^1$ is dense in $L^2$ and $\S$ and $\T$ are locally Lipschitz continuous in the $L^2$ norm, $\S$ and $\T$ have
unique continuous extensions $\St$ and $\Tt$
\\
\parbox{2.5in}
{ 
\begin{align*}
\St: L^2[0,Y] & \to L^2(\Om)
\\
\beta( \cdot) & \to \h(\cdot, \cdot)
\end{align*}
}
~~and~~
\parbox{2.5in}
{
\begin{align*}
\Tt : L^2[0,Y] & \to L^2[0,Y]
\\
\beta(\cdot) & \to h_3(z,z).
\end{align*}
}
\\
For $\beta \in  L^2[0,Y]$, $\St \beta$ is a candidate for a weak solution of 
(\ref{sidetc4})-(\ref{sidebc4}) and
$\Tt \beta$ is the candidate for the trace of this solution on $t=z$. Of course, we must first define what we mean by a 
weak solution of (\ref{sidetc4})-(\ref{sidebc4}).

For arbitrary $\h, \n \in C^1(\Om)$, from (\ref{eq:ulv}) and the divergence theorem  (see Figure \ref{CDest}) we have
\begin{align*}
\iint_{\D} (\L \n)^T \h + \n^T \L \h ~dx \, dt
=
\int_{AB} \n^T(cI-A) \h ~dt - \int_{OA} \n^T(I+A) \h ~dt  +  \int_{OB} \n^T A \h ~dt.
\end{align*}
Now $cI-A$ is a diagonal matrix with only the first, second and fourth diagonal entries being non-zero while $I+A$ is a 
diagonal matrix with only the first, third and fourth diagonal entries being non-zero. So keeping in mind (\ref{sidem}) the following seems an appropriate definition of a weak solution of (\ref{sidetc4})-(\ref{sidebc4}).
\begin{definition}
For arbitrary $\beta \in L^2[0,Y]$ and arbitrary $\a \in C^1[0,2Z]$ which satisfies (\ref{compat44}), (\ref{compat4}), 
we say that ${\bf h}\in L^2({\Om})$ is a 
weak solution of (\ref{sidetc4})-(\ref{sidebc4})  if (see Figure \ref{i12})
\begin{align}
\iint_{\D}(\n_t-A \n_z-\beta B \n)^T \h~dx \, dt=\int_{OB} \n^T A {\bf a} ~dt 
\label{def3}
\end{align}
for all $\n$ in 
\begin{align*}
\Lambda :=\{ \n \in C^1({\Om}) ~|~ n_1=0, (1-c)n_3+(1+c)n_4=0 ~ \text{on} ~ OA,~ n_1=n_2=n_4=0 ~ 
\text{on} ~ AB\}.
\end{align*}
\end{definition}

We now show the uniqueness and existence of the weak solution of  (\ref{sidetc4})-(\ref{sidebc4}).
\begin{proposition}[Weak solution of the sideways CBVP]\label{prop:weaksoln}
For any $\beta \in L^2[0,Y]$ and arbitrary $\a \in C^1[0,2Z]$ which satisfies (\ref{compat44}), (\ref{compat4}),
$\St \beta$ is the unique weak solution of (\ref{sidetc4})-(\ref{sidebc4}). Further, $\St \beta$ has an $L^2$ trace on $t=z$ 
which is $\Tt \beta$.
\end{proposition}
\begin{proof}
Given $\beta \in L^2[0,Y]$, we can find a sequence $\beta_k \in C^1[0,Y]$ which converges to $\beta$ in the $L^2$ norm.
Let $\h_k=\S \beta_k$ be the $C^1$ solution of (\ref{sidetc4})-(\ref{sidebc4}). Since $\S$ is locally Lipschitz, $\h_k$ will
 be a 
Cauchy sequence in $L^2(\Om)$ and hence has a limit $\h \in L^2(\Om)$; in fact this $\h$ defines $\St \beta$. 
Now $\h_k, \beta_k$ satisfy (\ref{def3}) for all $\n \in \Lambda$, so from the $L^2$ convergence it is clear that
(\ref{def3}) will hold for the $L^2$ limit of $\beta_k$ and $\h_k$. Hence $\St \beta$ is a weak solution of 
(\ref{sidetc4})-(\ref{sidebc4}). Further, the construction of $\St$ and $\Tt$ shows that $\Tt \beta$ is the trace of this 
solution on $t=z$.

It remains to prove the uniqueness of the weak solution. To show uniqueness it is enough to show that if $\h \in L^2(\Om)$,
$\beta \in L^2[0,Y]$ and
\[
\iint_{\D}(\n_t-A \n_z-\beta B \n)^T \h ~ dz \; dt  =0, \qquad \forall \n \in \Lambda
\]
then $\h=0$. 

From the density of $C^1$ in $L^2$, we can find sequences $\beta_k \in C^1[0,Y]$ and $\h^k \in C^1(\Om)$ whose $L^2$ 
limits are $\beta$ and $\h$ respectively. We show below that we can find $\n^k \in \Lambda$ such that 
$ \n^k_t - A \n^K_z - \beta_k B \n^k = \h^k$ in $\Om$ and further $\sup_{z \in [0,Y]} J(\n^k,z)$ is bounded above by a 
constant independent of $k$. 
Assuming this for the moment we have
\begin{align}
0 & = \iint_{\D} (\n^k_t - A \n^k_z - \beta B \n^k)^T \, \h ~ dz \, dt
\nn
\\
& = \iint_{\D} (\n^k_t - A \n^K_z - \beta_k B \n^k)^T \, \h ~ dz \, dt
+
\iint_{\D} (\beta - \beta_k) \, (B \n^k)^T \h  ~ dz \; dt
\nn
\\
& =  \iint_{\D} \h^T \h^k ~ dz \, dt + \iint_{\D} (\beta - \beta_k) \, (B \n^k )^T \h  ~ dz \; dt.
\label{eq:nhk}
\end{align}
Now (see Figure \ref{CDest})
\begin{align*}
\left |  \iint_{\D} (\beta - \beta_k) \,  (B\n^k)^T \h  ~ dz \; dt \right |
& 
\cleq \int_0^Y |(\beta - \beta_k)(z)| \, \int_{CD} \|\n^k(z,t)\| \, \|\h(z,t)\| \, dt
\\
& \cleq  \int_0^Y |(\beta - \beta_k)(z)| \, \sqrt{ J(\n^k,z) \, J(\h,z) } \, dz
\\
& \cleq \| \beta - \beta_k \|_{L^2[0,Y]} \, \left ( \int_0^Y J(\n^k,z) \, J(\h,z) \, dz \right )^{1/2}
\\
& \cleq  \| \beta - \beta_k \|_{L^2[0,Y]} \, \left ( \int_0^Y J(\h,z) \, dz \right )^{1/2}
\\
&=  \| \beta - \beta_k \|_{L^2[0,Y]} \, \|\h\|_{L^2(\Om)}
\end{align*}
which approaches $0$ as $k\to \infty$. Hence, from (\ref{eq:nhk}), taking the limit as $k \to \infty$, we obtain
$\|\h\|^2_{L^2(\Om)} =0$ and hence $\h=0$.

So it remains to show that if $\beta \in C^1[0,Y]$, then for any $\bold{F}(z,t) \in C^1(\Om)$ there is an $\n \in C^1(\Om)$ such that
(see Figure \ref{i56})
\begin{subequations}
\begin{align}
\n_t-A \n_z-\beta B \n &={\bf F} ~ \text{in} ~ {\Om}, \label{adjtt}\\
n_1=0,\quad (1-c)n_3+(1+c)n_4&=0 ~ \text{on} ~ OA,\label{adjmm}\\
n_1=n_2=n_4&=0 ~ \text{on} ~ AB; \label{adjbb}
\end{align}
\end{subequations}
further $\sup_{z \in [0,Y]} J(\n, z)$ is bounded above by a constant determined only by $\|\bold{F}\|_{L^2(\Om)}$, 
$\|\beta\|_{L^2[0,Y]}$, $Z$ and $c$.
\begin{figure}[!h]
\centering
\subfigure{
\epsfig{file=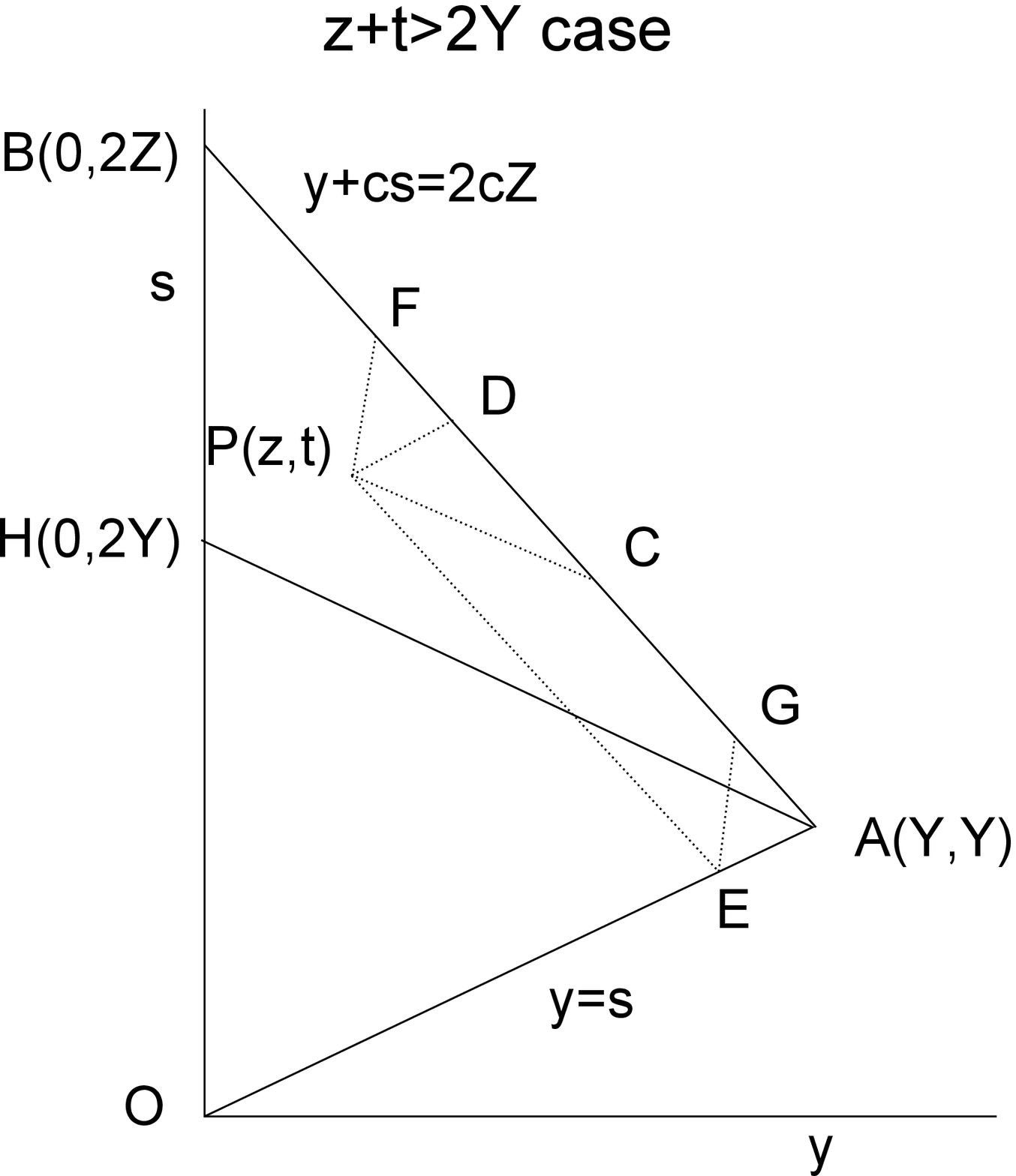, height=2.4in}
}
\hspace{0.3in}
\subfigure{
\epsfig{file=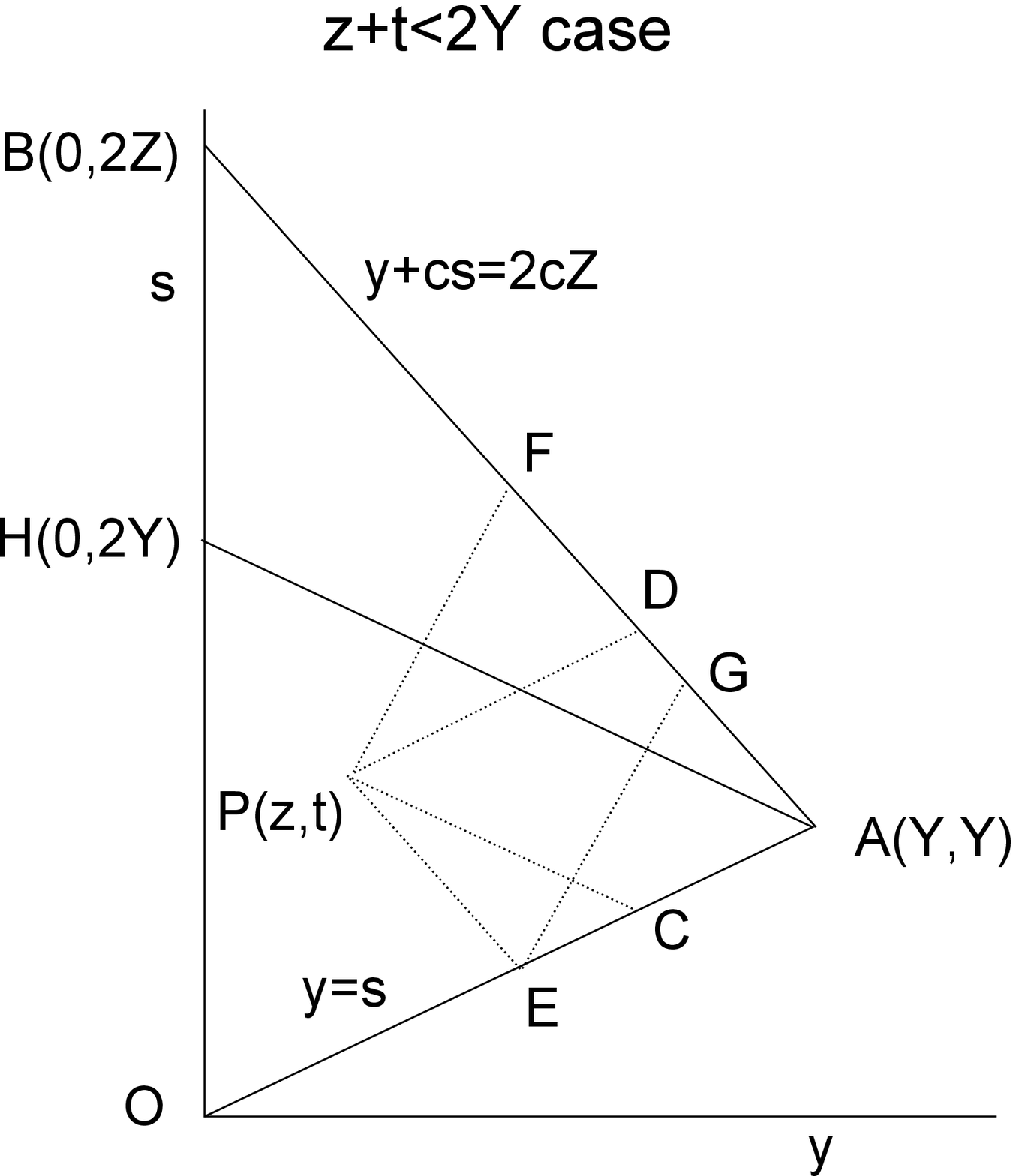, height=2.4in}
}
\caption{Rightward moving lines through $P(z,t)$ with slopes $\pm 1$ and $\pm 1/c$}
\label{i56}
\end{figure}

Define ${\bf r}({\bf n},z,t) :=\beta(z) B{\bf n}(z,t)+{\bf F}(z,t)$, and pick an arbitrary point $P(z,t)\in {\Om}$.
Integrating $(\ref{adjtt})$ along the characteristics and using the boundary conditions, we obtain the system of integral
equations
\begin{subequations}
\begin{align}
n_1(z,t)&=-\int_z^{y_C}r_1({\bf n},y,z+t-y)~dy, \label{vvv1v} \\
n_2(z,t)&=-\int_z^{y_D}r_2({\bf n},y,y+t-z)~dy,\\
n_3(z,t)&=- \frac{1}{c}\int_z^{y_E}r_3\left({\bf n},y,\frac{z+ct-y}{c}\right)~dy+
\frac{1+c}{c(1-c)}\int_{y_E}^{y_G}r_4\left({\bf n},y,\frac{y+cy_E-y_E}{c}\right)~dy, \label{vvv3} \\
n_4(z,t)&=\frac{1}{c}\int_z^{y_F}r_4\left({\bf n},y,\frac{y+ct-z}{c}\right)~dy.\label{vvv4v}
\end{align}
\end{subequations}
The existence of a $C^1$ solution of this system of integral equations holds by the usual argument.
The upper bound on $J(\n,z)$ may be obtained by using arguments similar to those used in the proof of Lemma 
\ref{lemma:energy}. The only difference is that the identity (\ref{LRHS}) must be integrated over the region CAD
 (see Figure \ref{CDest}) instead of the region OCDB and one should now choose $\ep = (1+c)^4/(1-c)^3.$

\end{proof}


\subsection{Local Reconstruction}

Suppose $\beta \in \dot{C}^1[0,2Z]$ and $\m(z,t)$ is the $C^1$ solution of (\ref{compt}) - (\ref{compb}). Given $m_1(0,t), 
m_3(0,t)$ on $[0,2Z]$, our goal is to reconstruct $\beta$ on $[0,Y]$. The reconstruction will occur piece by piece, first over
an interval $[0,\delta]$, then over $[\delta, 2 \delta]$, then $[2 \delta, 3 \delta]$, and so on, with the $\delta>0$ 
determined by  the value of $m_1(0,t), m_3(0,t)$ on $[0,2Z]$ . The crux of the global reconstruction is a local
reconstruction as described next.
\begin{figure}[!h]
\centering
\epsfig{file=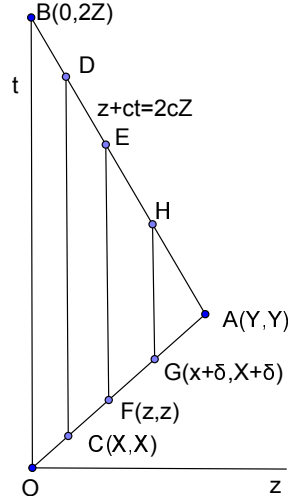, height=2.6in}
\caption{Local Reconstruction}
\label{mm5}
\end{figure}
Suppose $0 \leq X < Y$; given $\beta(X)$ and $\m(X, t)$ for all $t \in [X, (2cZ - X)/c]$ (that is given $\m$ on CD -
 see figure \ref{mm5}), we reconstruct $\beta$ on the interval $[X, X+\delta]$ for some $\delta>0$.

For arbitrary $X \in [0,Y]$ and $\delta>0$ such that $X+\delta\le Y$, define
\begin{align*}
{\tilde D}_{X,\delta} &:=\{(z,t)~| ~(z,t)\in {\tilde D},  X \le z\le X+\delta \}.
\end{align*}
Further, for any $K_X>0$, define the complete metric space (in the $L^2$ norm)
\begin{align*}
\Bb_X :=\{\beta \in L^2[X,X+\delta] ~|~ ||\beta||^2_{L^2[X,X+\delta]}\le K_X\},
\end{align*}
which, for an arbitrary fixed real number $\beta_*$, has a dense subset
\begin{align*}
\B_X :=\{\beta \in C^1[X,X+\delta] ~|~ ||\beta||^2_{L^2[X,X+\delta]}\le K_X, ~ \beta(X) = \beta_*\}.
\end{align*} 

For an arbitrary $\a(\cdot) \in C^1[X, (2cZ-X)/c]$, consider the CBVP
\begin{subequations}
\begin{align}
\h_t &= A\h_z + \beta B \h, ~ \qquad \text{in} ~ {\tilde D}_{X,\delta},  \label{subsidet}\\
h_3(z,z)&=\frac{(c-1)^2}{(c+1)^2}h_4(z,z),\qquad  X\le z\le X+\delta, \label{subsidem}\\
\h(X,t) &=\a(t), \qquad t \in [X, (2cZ-X)/c]. \label{subsideb}
\end{align}
\end{subequations}
For $\beta \in C^1[X,X+\delta]$, just as in Proposition \ref{prop:sideways}, one can verify that the matching conditions
on $\a(\cdot)$ needed for a $C^1$ solution of (\ref{subsidet})-(\ref{subsideb}) are
\begin{subequations}
\begin{align}
(c+1)^2  a_3(X) & =(1-c)^2a_4(X),\label{xmatch1}\\
(c+1)^2\left((1+c)a_3'(X)- \beta(X) (B\a)_3(X) \right) 
& =  (c-1)^2  \left( (c-1)a_4'(X)+ \beta(X) (B \a)_4(X) \right). \label{xmatch2}
\end{align}
\end{subequations}
Here is the important local reconstruction step.
\begin{proposition} \label{localRec}
Let $0\le X\le Y$, $\beta_* \in \R$ and $\a \in C^1[X, (2cZ-X)/c]$ such that (\ref{xmatch1}), (\ref{xmatch2}) hold.
There exists a $\delta>0$, $K_X>0$ and a unique $\beta \in \Bb_X$ with
\begin{align}
h_3(z,z)=\frac{c-1}{2(c+1)}\beta(z), \quad X\le z\le X+\delta,
\label{eq:h3cc}
\end{align}
where ${\bf h}(z,t)\in L^2({\tilde D}_{Z,\delta})$ is the unique weak solution of  (\ref{subsidet})-(\ref{subsideb}).
Actually, it is sufficient to choose any $K_X>0$ and $\delta>0$ so that
\begin{align}
K_X\ge \frac{8(1+c)^2}{(1-c)^2}J_X ,\quad  \delta \leq \text{min}\left(Y-X, \frac{c^2\epsilon^2}{256K_X}\right), 
\label{mindelta}
\end{align}
where $J_X= \|a_1\|^2_{L^2}+ \|a_2\|^2_{L^2}+c \|a_3\|^2_{L^2}+c \epsilon \|a_4\|^2_{L^2}$ and 
$\epsilon = \frac{c(1-c)^3}{(1+c)^4}$.
\end{proposition}

\begin{proof}
If $\beta \in C^1[X, X+\delta]$ then (extending $\beta$ arbitrarily to a function in $C^1[X,Y]$), from 
Proposition \ref{prop:sideways}, (\ref{subsidet})-(\ref{subsideb}) has a unique solution
$\h \in C^1[\D_{X,\delta}]$. So we may define the map 
\[
Q : \beta \to \frac{2(c+1)}{c-1}h_3(z,z)
\]
and our goal is to find a fixed 
point for this map. We do so by setting up $Q$ as a contraction map on a complete metric space.

First we show that for appropriate $\delta>0$ and $K_X>0$, if $\beta \in \B_X$ then $Q \beta \in \B_X$.
Given $\beta \in \B_X$, let $\h$ be the corresponding unique $C^1$ solution of (\ref{subsidet})-(\ref{subsideb}).
From Lemma 
\ref{lemma:energy}  (using $\L \h=0$ and letting $\lambda \to \infty$) we have  (see Figure \ref{mm5})
\begin{align*}
J(\h,z)+\int_{CF} & (2h_1^2+(1+c)h_3^2-\epsilon(1-c)h_4^2)~dy
\\ 
& \le J_X +\frac{4}{c\epsilon}\int_X^z |\beta(y)| \, J(\h,y)~dy,\quad z\in [X,X+\delta].
\end{align*}
If we take $\epsilon = \frac{c(1-c)^3}{(1+c)^4}$ and use the characteristic condition (\ref{subsidem}) then 
\begin{align*}
\int_{CF}((1+c)h_3^2-\epsilon (1-c)h_4^2)~dt=\int_{CF}h_3^2~dt 
\end{align*}
and hence
\begin{align}
J({\bf h},z)+\int_{CF}h_3^2\le J_X+\frac{4}{c\epsilon}\int_X^z|\beta(y)|J({\bf h},y)~dy,\quad z\in [X,X+\delta];
\label{also4}
\end{align}
so from Gronwall's inequality (for use later)
\begin{align}
J({\bf h},z)\le e^{4\sqrt{K_X\delta}/(c\epsilon)}J_X,\quad z\in [X,X+\delta]. \label{jhz4}
\end{align}
Define $J^*({\bf h}) :=\displaystyle \max_{z\in [X,X+\delta]}J({\bf h},z)$; then from $(\ref{also4})$ we have
\begin{align*}
J^*({\bf h})+\int_{CG}h_3^2~dz\le
2J_X+\frac{8}{c\epsilon}\int_X^{X+\delta}|\beta(y)|J^*({\bf h})~dy
\le 
2J_X+\frac{8\sqrt{\delta K_X}}{c\epsilon}J^*({\bf h}),
\end{align*}
which implies that
\begin{align*}
\left(1-\frac{8\sqrt{\delta K_X}}{c\epsilon}\right)J^*({\bf h})+\int_{CG}h_3^2~dz&\le 2J_X.
\end{align*}
So if $\delta\le \frac{c^2\epsilon^2}{64K_X}$, we have
\begin{align*}
||Q\beta||^2_{L^2[X,X+\delta]}=\frac{4(1+c)^2}{(1-c)^2}\int_X^{X+\delta}h_3(z,z)^2~dz
\le \frac{8(1+c)^2}{(1-c)^2}J_X \le K_X,
\end{align*}
if we choose $K_X \ge \frac{8(1+c)^2}{(1-c)^2}J_X$.

Summarizing, if we chose 
\beqn
\ep = \frac{c(1-c)^3}{(1+c)^4},  \qquad K_X \ge \frac{8(1+c)^2}{(1-c)^2}J_X, \qquad
\delta\le \frac{c^2\epsilon^2}{64K_X},
\label{eq:con1}
\eeqn
then we have a map
\begin{align*}
Q: \B_X&\mapsto \B_X,\\
(Q\beta)(z)&=\frac{2(c+1)}{c-1}h_3(z,z).
\end{align*}
We now show that $Q$ is a contraction if $\delta>0$ is small enough.

Suppose $\beta, \tilde{\beta} \in \B_X$ and let ${\bf h}, {\bf \tilde h}$ be the corresponding unique $C^1$ solutions of 
$(\ref{subsidet})-(\ref{subsideb})$. Define ${\bf p} :={\bf h}-{\bf \tilde h}$; then 
${\cal L} ({\bf p}) =(\beta-{\tilde \beta})B{\bf \tilde h}$, $\p=0$ on $z=X$, and $p_3, p_4$ satisfy (\ref{subsidem}) on $z=t$. 
Apply Lemma \ref{lemma:energy} but on the interval $[X, X+\delta]$ instead of the interval $[0,Y]$,
and choose $\epsilon = \frac{c(1-c)^3}{(1+c)^4}$. Noting that (as vector norms) $\|Bh\|^2 \leq 4 \|h\|^2 $ and
using (\ref{jhz4}) for $\tilde{h}$,
we obtain
\begin{align*}
J^*({\bf p}) + \int_{CG}p_3^2~dy & 
\le \frac{8 \lambda}{c \ep} \int_X^{X+\delta} (\beta-{\tilde \beta})^2(y)  J({\bf \tilde h},y) ~dy 
+ \frac{2}{c\epsilon}\int_X^{X+\delta} \left(4|\beta(y)|
+ \frac{1}{\lambda}\right) J^*({\bf p}) ~ dy  \nn 
\\
& \le \frac{8 \lambda  e^{4\sqrt{K_X\delta}/(c\epsilon)}J_X }{c \ep}        \int_X^{X+\delta}(\beta-{\tilde \beta})^2(y)~dy 
+ \frac{2\delta / \lambda+8\sqrt{\delta K_X}}{c\epsilon}J^*({\bf p}) 
\end{align*}
which implies that
\begin{align*}
\left(1-\frac{2\delta / \lambda+8\sqrt{\delta K_X}}{c\epsilon}\right)J^*({\bf p}) +  \int_{CG}p_3^2~dy
&\le  \frac{8 \lambda  e^{4\sqrt{K_X\delta}/(c\epsilon)}J_X }{c \ep}   \int_X^{X+\delta}(\beta-{\tilde \beta})^2(y)~dy.
\end{align*}
So choosing
\beqn
\delta \le \min\left(\frac{c\lambda\epsilon}{4},\frac{c^2\epsilon^2}{256K_X}\right),
\label{eq:con2}
\eeqn
we have
\begin{align*}
\int_{CG}p_3^2~dy
&\le  \frac{8 \lambda  e^{4\sqrt{K_X\delta}/(c\epsilon)}J_X }{c \ep}   \int_X^{X+\delta}(\beta-{\tilde \beta})^2(y)~dy
\end{align*}
which implies
\beqn
\|Q \beta - Q \tilde{\beta}\|_{L^2[X, X+\delta]}^2
\leq \sigma \| \beta - \tilde{\beta} \|_{L^2[X, X+\delta]}^2
\label{eq:contraction}
\eeqn
where
\[
\sigma = \frac{4 (1+c)^2}{ (1-c)^2} \, \frac{8 \lambda  e^{4\sqrt{K_X\delta}/(c\epsilon)} J_X }{c \ep}.
\]

The constraints on $K_X, \delta$ and $\lambda$ are given by (\ref{eq:con1}), (\ref{eq:con2}). We take 
\beqn
\lambda = \frac{c\epsilon(1-c)^2}{64(1+c)^2J_X} e^{- 4\sqrt{K_X\delta}/(c\epsilon)}
\label{eq:lambda}
\eeqn
then $\sigma = 1/2$; so we now have to choose $\delta>0$ small enough so that (\ref{eq:con1}) and (\ref{eq:lambda}) 
imply (\ref{eq:con2}). Some calculations\footnote{
Now from the last inequality in (\ref{eq:con1}) we have 
$
\dfrac{4 \sqrt{ K_X \, \delta}}{ c \ep} \leq \dfrac{1}{2};
$
hence, using  (\ref{eq:lambda}) and the second inequality in (\ref{eq:con1}), we have
\[
\lambda \geq \frac{c\epsilon(1-c)^2}{64(1+c)^2J_X} e^{-1/2} \geq \frac{c\epsilon(1-c)^2}{8(1-c)^2K_X} e^{-1/2}
\geq \frac{c \ep}{24 K_X}.
\]
So $ \dfrac{c \lambda \ep}{4} \geq \dfrac{c^2 \ep^2}{96 K_X}$ and hence (\ref{eq:con2}) holds if 
$ \delta \leq \dfrac{c^2 \ep^2}{256 K_X}$.
}
will show that choosing any $\delta$ with
\[
\delta \leq \text{min}\left(Y-X, \frac{c^2\epsilon^2}{256K_X}\right)
\]
will satisfy (\ref{eq:con1}), (\ref{eq:con2}). Hence choosing $K_X$ and $\delta>0$ which satsify (\ref{mindelta}), we 
have shown that $Q$ is a contraction map with $\sigma=1/2$ in (\ref{eq:contraction}).

So $Q$ has an extension $\Qt$ to the complete metric space $\Bb_X$, namely
\begin{align*}
\Qt: \Bb_X&\mapsto \Bb_X,\\
(\Qt \beta)(z)&=\frac{2(c+1)}{c-1}h_3(z,z)
\end{align*}
where $\h=\St \beta$ and $\h$ has an $L^2$ trace on $z=t$ (because of Proposition \ref{prop:weaksoln}). Further, $\Qt$ 
will also be a contraction map, and hence have a unique fixed point, which may be obtained by an algorithm.

\end{proof}

%
\subsection{Global reconstruction}
\noindent We defined the forward map $\F$
\begin{align*}
\F: \dot{C}^1[0,Z]&\mapsto C^1[0,2Z]\times C^1[0,2Z],\\
(\F\beta)(z)&=[m_1(0,t), m_3(0,t)]
\end{align*}
where ${\bf m}$ is the solution of $(\ref{compt})-(\ref{compb})$ and we have shown in Theorem \ref{stab} that $\F$ is 
injective. 

Since $[\phi(\cdot), \psi(\cdot)]$ are in the range of $\F$, there is a unique (unknown) $\beta(\cdot) \in \CdZ$ and a
corresponding unique (unknown) $C^1$ solution $\m$ of (\ref{compt})-(\ref{compb}) such that $m_1(0, t) = \phi(0,t)$,
$m_3(0,t) = \psi(0,t)$, $t \in [0,2Z]$. As per the hypothesis, we also assume that $\|\beta\|^2_{L^2[0,Y]} 
\leq K$ for some known $K \geq 0$, for this unique unknown $\beta$.

Applying Proposition \ref{localRec} with $X=0$, $\beta_*=0$, $\a=[\phi,0,\psi,0]$ (note that the $C^1$ matching conditions 
(\ref{xmatch1}), (\ref{xmatch2}) hold because we already know the existence of a $C^1$ solution, namely $\m$), 
we can find a $\delta_0>0, K_0>0$ and a unique
$\beta(\cdot) \in L^2[0,\delta]$ such that (\ref{eq:h3cc}) holds where $\h(z,t)$ is the solution of 
(\ref{subsidet})-(\ref{subsideb}). This $\beta$ must be same as $\F^{-1}[\phi,\psi]$ because the $\m$ corresponding
to $\F^{-1}[\phi,\psi]$ already satisfies (\ref{subsidet})-(\ref{subsideb}). Since Proposition \ref{localRec} was constructive, 
we have recovered $\beta$ on the interval $[0, \delta_0]$. Further, from Prop \ref{prop:sideways} applied to the interval
$[0, \delta_0]$ with $\a = [\phi, 0, \psi,0]$ we can can construct the unique $C^1$ solution of 
(\ref{sidetc4})-(\ref{sidebc4}) on the region $\tilde{D}_{X, \delta}$, which is the 
$\m$ corresponding to the $\beta= \F^{-1}[\phi, \psi]$.
Hence we now also have $\m(\delta_0,\cdot)$ on the interval $[\delta_0, (2cZ- \delta_0)/c]$ as well as $\beta(\delta_0)$.

Now we show the general step. Suppose, for some $X>0$ we are given $\m(X,\cdot)$  on the interval 
$[X, (2cZ- X)/c]$ as well as $\beta(X)$. Then the $C^1$ matching conditions are automatically satisfied and hence an 
application of Proposition \ref{localRec} with $\a(\cdot)= \m(X, \cdot)$, there exists a $\delta>0$ such that we can 
recover $\beta( \cdot)$ on the interval $[X, X+\delta]$. Then repeating the argument in the previous paragraph we can calculate $\m(X+\delta,t)$ for all $t \in [X+\delta,  (2cZ-X - \delta)/c]$. 

We can then apply this process repeatedly. To show that this process will end in a finite number of steps, we need to obtain 
a lower bound on the step size $\delta$ guaranteed by Proposition \ref{localRec}. Let $\beta = \F^{-1}[\phi, \psi]$ 
and $\m$ 
the corresponding solution of (\ref{compt})-(\ref{compb}); note that $\beta$ and $\m$ are unknown but
$\m(0,t)= [\phi(t), 0, \psi(t),0]$, $t \in [0,2Z]$ is given to us and $\|\beta\|^2_{L^2[0,Y]} 
\leq K$ for some known $K \geq 0$. Applying Lemma \ref{lemma:energy} for this $\beta$ and $\m$, (\ref{jmz04}) implies 
that
\beqn
J(\m,z) \leq e^{ 4 \sqrt{KY}/(c \ep)}, \qquad z \in [0,Y]
\label{eq:J}
\eeqn
where $\ep = c(1-c)^3/(1+c)^4$. Now at each iteration step we applied Proposition \ref{localRec} with
$\a = \m$, so from (\ref{eq:J})
\[
J_X = J(\m,X) \leq  e^{ 4 \sqrt{KY}/(c \ep)}
\]
and hence if we take 
\[
K_X = \frac{8(1+c)^2}{(1-c)^2}J_X
\]
then
\[
\frac{c^2\epsilon^2}{256K_X} = \frac{ c^2 (1-c)^2 \ep^2}{2048 (1+c)^2 J_X}
\geq  \frac{ c^2 (1-c)^2 \ep^2}{2048 (1+c)^2 } e^{-4 \sqrt{KY}/(c \ep)} =: \delta_*.
\]
So at every step we can choose a step size $\delta_*$ independent of $X$, except for the last step when the step size will be $\min( Y-X, \delta_*)$.

%

\section{Numerical reconstruction}
We now show the results from a numerical implementation of the scheme suggested by the proof of Theorem \ref{thm:recon}. 
The 
proof involved the construction of a fixed point for a contraction map $Q$; the fixed point is the limit of the sequence
$\beta_n$ where $\beta_0$ is chosen arbitrarily and $\beta_{n+1} = Q \beta_n$.

The data for the inverse problem, for the chosen $\beta$, was generated by solving the CBVP (\ref{compt})-(\ref{compb}) 
using the Crank-Nicolson method with interpolation to solve the ODE along the characteristics. The solution of the
inverse problem requires solving the sideways CBVP (\ref{eq:hmsde})-(\ref{eq:hmscc}), again using the Crank-Nicolson method 
with interpolation to solve the ODE along the characteristics. In the examples below we used
\[
c=0.5, ~~ Z = \frac{\pi}{2}
\]
and $N$ represents the number of subdivisions of $[0,Z]$.
If $\beta$ is the exact value and $\beta_{app}$ the numerical approximation from our inversion then we plot the $L^2$ error
$E_2$  and the relative $L^\infty$ error $E_\infty$ to judge the effectiveness of the algorithm where
\[
E_2 = \left ( \frac{Z}{N} \sum_{i=1}^N (\beta - \beta_{app})^2(z) \right )^{1/2},
\qquad
E_\infty = \max_{\beta(z_i) \neq 0} \left | \frac{(\beta - \beta_{app})(z_i)} {\beta(z_i)} \right |.
\]
 In the examples below, the calculated $\beta$ and the exact $\beta$ are very close compared to the scale 
 so we see only only graph even though we have drawn two.

We apply the inversion scheme to four examples and we start with a simple example with just a little bit of oscillation.

\noindent
{\bf Example 1}
Here $\beta(z)=3 z^2 \, \cos (10 z) \, \log(z+1)$, we use an initial guess $\beta_0(z) =z$ and $N=2^9$. The 
iterations converged in $17$ steps. Figures  \ref{error-smooth} and \ref{exact-smooth} show the accuracy of our 
reconstruction.
\begin{figure}[!h]
\centering
\includegraphics[trim=5mm 20mm 5mm 10mm, clip=true, scale=.4]{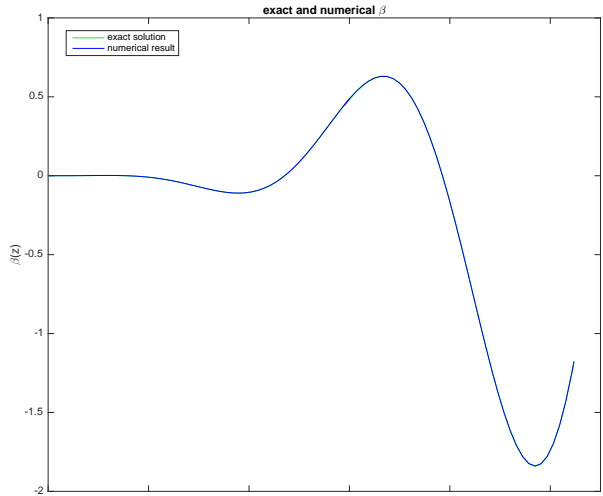}
\caption{Comparing exact $\beta$ with reconstructed $\beta$}
\label{exact-smooth}
\end{figure}
\vspace{0.2in}
\begin{figure}[!h]
\centering
\includegraphics[trim=0mm 0mm 0mm 0mm, clip=true, scale=.5]{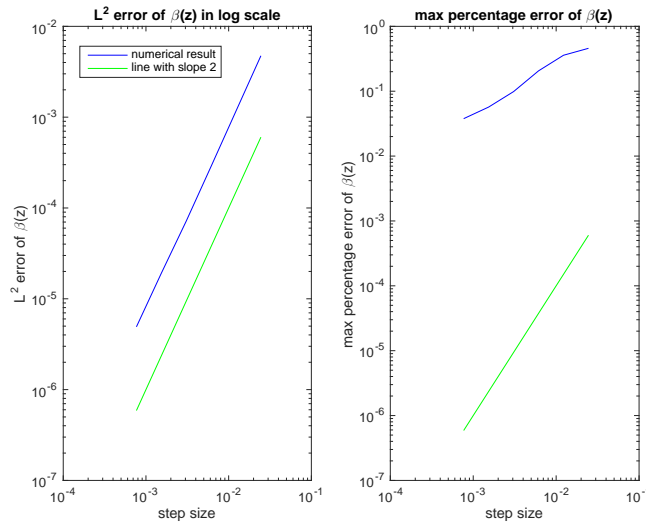}
\caption{$L^2$ error and relative $L^\infty$ error }
\label{error-smooth}
\end{figure}

\noindent
{\bf Example 2}
Here $\beta(z)=z\sin(100z)\log(z+1)$, 
 an initial guess $\beta_0(z) =z$ and $N=2^6, 2^7, \cdots, 2^{11}$. In all cases, the iterations 
converged in $17$ steps and Figures \ref{sine} and \ref{sineerror} reflect the accuracy of our reconstruction.

\begin{figure}[!h]
\centering
\includegraphics[trim=1mm 1mm 1mm 1mm, clip=true, scale=.5]{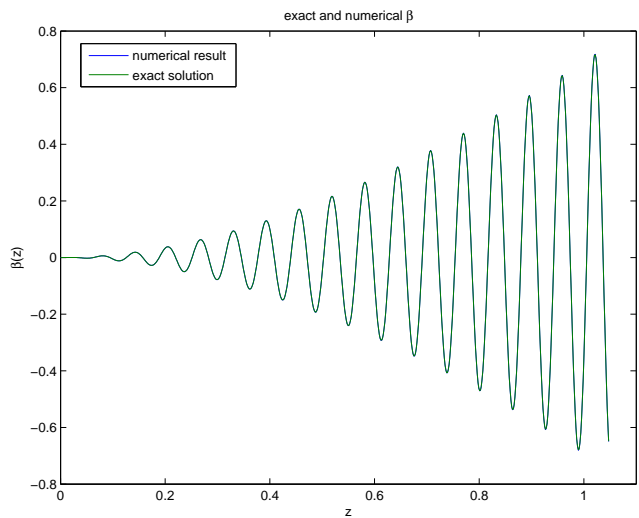}
\caption{Comparing exact $\beta$ with reconstructed $\beta$}
\label{sine}
\end{figure}

\begin{figure}[!h]
\centering
\includegraphics[trim=1mm 1mm 1mm 1mm, clip=true, scale=.5]{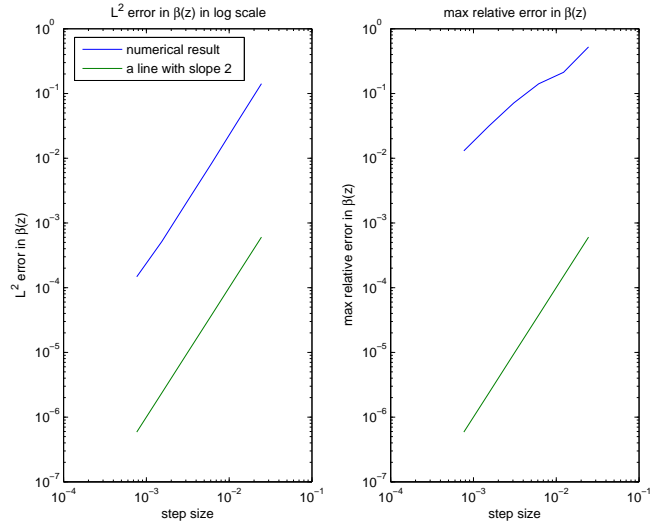}
\caption{$L^2$ error and relative $L^\infty$ error }
\label{sineerror}
\end{figure}
\vspace{0.2in}

\noindent{\bf Example 3} 
Here $\beta(z)=9z^2\cos(100z)\log(z+1)$, an initial guess $\beta_{0}(z)=z$, 
and $N=2^6, 2^7, \cdots, 2^{11}$. In all cases, the iterations 
converged in $14$ steps and Figures \ref{coserror} and \ref{cos} reflect the accuracy of our reconstruction.

\begin{figure}[!h]
\centering
\includegraphics[trim=1mm 1mm 1mm 1mm, clip=true, scale=.5]{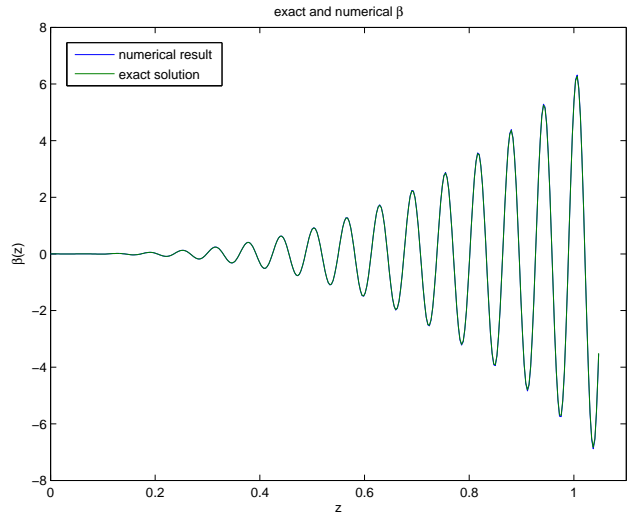}
\caption{Comparing exact $\beta$ with reconstructed $\beta$}
\label{coserror}
\end{figure}
\vspace{0.1in}

\begin{figure}[!h]
\centering
\includegraphics[trim=1mm 1mm 1mm 1mm, clip=true, scale=.5]{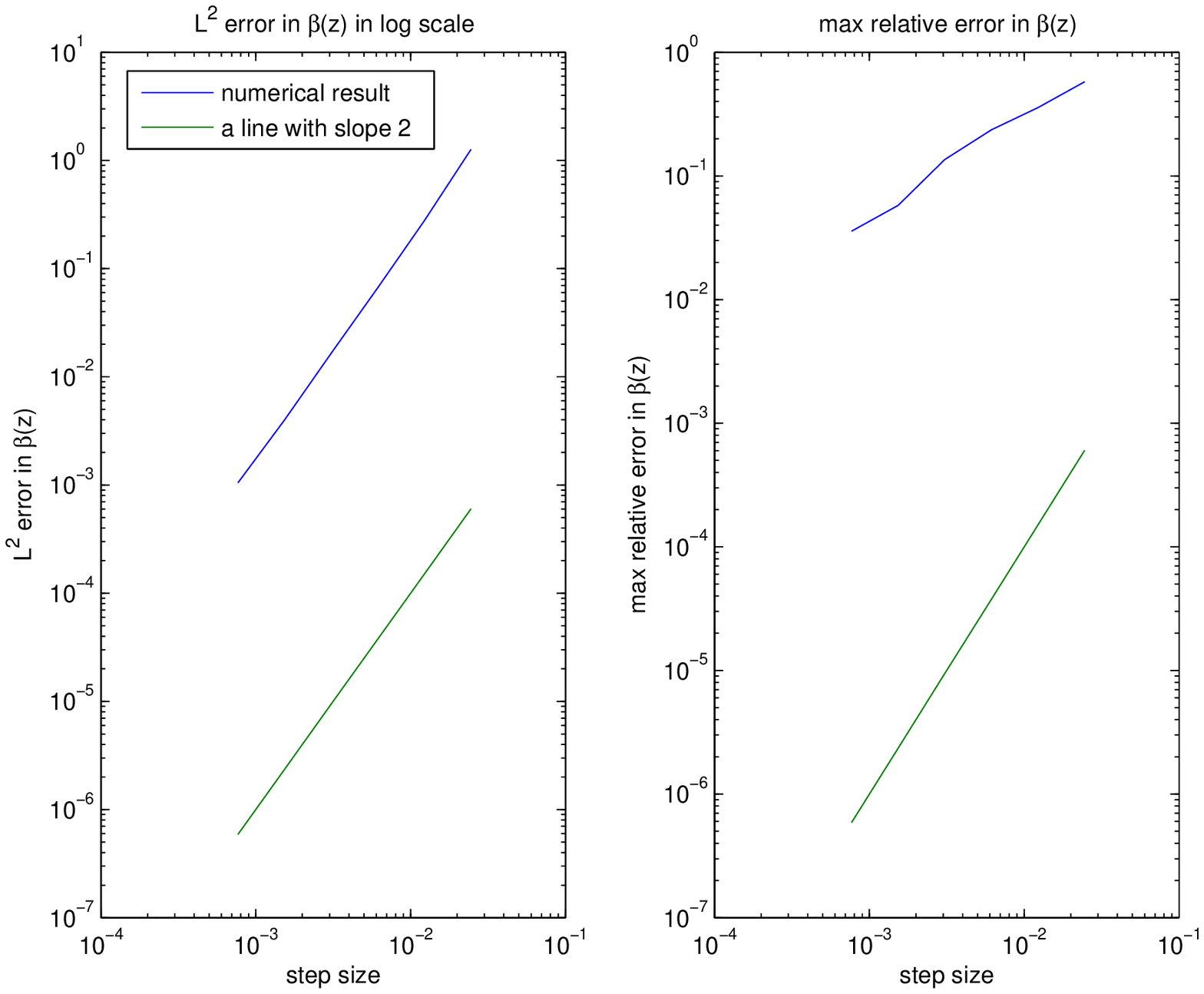}
\caption{$L^2$ error and relative $L^\infty$ error }
\label{cos}
\end{figure}

\noindent
{\bf Example 4}
Here $\beta(z)=z\sin(100z)e^{az}$ where $a$ is an integer, an initial guess $\beta_0(z)=z$. For the algorithm to 
converge $N$ had to be increased as $a$ increased  - see Table \ref{exp}.
\begin{table}[!h]
\centering
\begin{tabular}{|c|c|c|c|c|c|c|}
\hline
$a$ & $3$ & $4$ & $4$ & $5$ & $5$ & $6$ \\
\hline
N & $2^6$ & $2^7$ & $2^8$ & $2^9$ & $2^{10}$ & $2^{11}$\\
\hline
\end{tabular}
\caption{N value for the algorithm to converge}
\label{exp}
\end{table}

\section{Derivation of the model}\label{sec:model}
 
During the 2000 Mathematical Problems in Industry workshop at the University of Delaware, Greg Luther, then of Corning Inc., 
proposed the problem of modeling the twist in a birefringent optical fiber and determining this twist from the response of the 
fiber, measured at one end of the fiber, to an impulsive source applied at the same end of the fiber. He suggested \cite{corsi98}, 
\cite{corsi99} as possible sources for information. A few months after the workshop, a model was proposed in \cite{mpi2000}. 
Since this derivation is not readily available, it is included here.

Consider an optical fiber stretching along the $z$ axis and let $\E(z,t), \P(z,t)$ be the electric
field and the polarization at the point $z$ units away from the left end of the fiber; then
$\E$ and $\P$ obey Maxwell's equations
\begin{equation}
 \nabla^2 \textbf{E}-\nabla(\nabla \cdot \textbf{E})=
\frac{1}{c_0^2} \textbf{E}_{tt} +\frac{1}{\epsilon_0c_0^2} \textbf{P}_{tt} \label{maxwell}
\end{equation}
where $c_0$ is the speed of light in vacuum and $\epsilon_0$ is the permittivity of free space. Assume that
${\bf E}$ and ${\bf P}$ have no component along the fiber; since $\E$ and $\P$ depend only on $z$ and $t$, 
$(\ref{maxwell})$ reduces to
\begin{equation}
 \textbf{E}_{zz}=\frac{1}{c_0^2}\textbf{E}_{tt}+\frac{1}{\epsilon_0c_0^2}\textbf{P}_{tt}. \label{epmax}
\end{equation}
At every point in the fiber, there are two unit orthogonal vectors $\textbf{v}_1(z)$ and $\textbf{v}_2(z)$ perpendicular 
to the fiber, which represent the polarization directions of the two channels in the fiber. As the fiber twists along its length, 
the polarization directions change. Since $\textbf{v}_1$ and $\textbf{v}_2$ are orthogonal unit vectors in a plane perpendicular 
to the fiber, $d\textbf{v}_1/dz$ is orthogonal to $\textbf{v}_1(z)$ and hence $\dfrac{d\textbf{v}_1}
{dz}=\beta(z)\textbf{v}_2$ for 
some real valued function $\beta(z)$ and one may then show that $\dfrac{d\textbf{v}_2}{dz}=-\beta \textbf{v}_1$.

Since $\E(z)$ has no component along the fiber, we may write $\textbf{E}=E_1\textbf{v}_1+E_2\textbf{v}_2$. Further 
we assume that the polarization vector $\P$ is related to the electric field $\E$ via
\begin{align*}
 \textbf{P}=\epsilon_0(\alpha_1E_1\textbf{v}_1+\alpha_2E_2\textbf{v}_2)
\end{align*}
where $\alpha_1, \alpha_2$ are real constants. Substituting these representations of $\E$ and $\P$ into (\ref{epmax}), 
using the relations for the derivatives of $\v_1$ and $\v_2$, and matching the $\v_1$ and $\v_2$ components we obtain
\begin{subequations}
\begin{align}
 (E_{1z}-\beta E_2)_z-\beta (E_{2z}+\beta E_1)=\frac{1}{c_1^2}E_{1tt}, \label{ee1}\\
 (E_{2z}+\beta E_1)_z+\beta (E_{1z}-\beta E_2)=\frac{1}{c_2^2}E_{2tt} \label{ee2}
\end{align}
\end{subequations}
where it is assumed that $1+\alpha_i>0$ and we define $c_i = \dfrac{c_0}{ \sqrt{1+ \alpha_i}}$.

The second order hyperbolic system of equation (\ref{ee1}), (\ref{ee2}) has two speeds of propagation $c_1, c_2$ and
$E_1$, $E_2$ are, respectively, the waves propagating at these speeds. We rewrite this system as a first order system 
where we
distinguish between the right and left moving components of these waves. If we define $\M = [M_1, M_2, M_3, M_4]^T$
where
\begin{subequations}
\begin{align}
2M_1 &=E_{1z}-\beta E_2+\frac{1}{c_1}E_{1t}, \quad 2M_2 =E_{1z}-\beta E_2-\frac{1}{c_1}E_{1t}, \label{ME12}\\
2M_3 &=E_{2z}+\beta E_1+\frac{1}{c_2}E_{2t}, \quad 2M_4 =E_{2z}+\beta E_1-\frac{1}{c_2}E_{2t} \label{ME34}
\end{align}
\end{subequations}
then one may verify that $\M(z,t)$ satisfies (\ref{IVP11}); here, WLOG (because of scaling), for convenience we have 
assumed that the faster speed $c_1=1$ and the smaller speed $c_2=c$.

%

\bibliography{references}
\bibliographystyle{plain}

\end{document}